\documentclass[11pt,reqno,twoside]{article}
\usepackage{mysty}

\usepackage{wrapfig}

\usepackage{patchcmd}
\makeatletter
\patchcommand\@starttoc{\begin{quote}}{\end{quote}}
\makeatother

\usepackage{enumitem}

\SetLabelAlign{parright}{\parbox[t]{\labelwidth}{\raggedleft{#1}}}
\setlist[description]{style=multiline,topsep=4pt,align=parright}

\makeatletter
\let\reftagform@=\tagform@
\def\tagform@#1{\maketag@@@{(\ignorespaces\textcolor{black}{#1}\unskip\@@italiccorr)}}
\newcommand{\iref}[1]{\textup{\reftagform@{\tcr{\ref{#1}}}}}
\makeatother

\begin{document}

\title{Local Convergence Properties of SAGA/Prox-SVRG and Acceleration}
\author{Clarice Poon\thanks{Equal contributions.}~~\thanks{DAMTP, University of Cambridge, Cambridge, UK. E-mail: C.M.H.S.Poon@maths.cam.ac.uk.}
\and
Jingwei Liang\samethanks[1]\thanks{DAMTP, University of Cambridge, Cambridge, UK. E-mail: jl993@cam.ac.uk.}
\and
Carola-Bibiane Sch{\"{o}}nlieb\thanks{DAMTP, University of Cambridge, Cambridge, UK. E-mail: cbs31@cam.ac.uk.}}
\date{}
\maketitle

\begin{abstract}
Over the past ten years, driven by large scale optimisation problems arising from machine learning, the development of stochastic optimisation methods have witnessed a tremendous growth. 
However, despite their popularity, the theoretical understandings of these methods are quite limited in contrast to the deterministic optimisation methods. 
In this paper, we present a local convergence analysis for a typical type of stochastic optimisation methods: proximal variance reduced stochastic gradient methods, and mainly focus on the SAGA \cite{saga14} and Prox-SVRG~\cite{proxsvrg14} algorithms. 
Under the assumption that the non-smooth component of the optimisation problem is partly smooth relative to a smooth manifold, we present a unified framework for the local convergence analysis of the SAGA/Prox-SVRG algorithms: (i) the sequences generated by the SAGA/Prox-SVRG are able to identify the smooth manifold in a finite number of iterations; (ii) then the sequence enters a local linear convergence regime. 
Beyond local convergence analysis, we also discuss  various possibilities for accelerating these algorithms, including adapting to better local parameters, and applying higher-order deterministic/stochastic optimisation methods which can achieve super-linear convergence. 
Concrete examples arising from machine learning are considered to verify the obtained results. 
\end{abstract}



\section{Introduction}\label{sec:introduction}

\subsection{Non-smooth optimisation}

Modern optimisation has become a core part of many fields in science and engineering, such as machine learning, inverse problem and signal/image processing, to name a few. 
In a world of increasing data demands, there are two key driving forces behind modern optimisation.
\begin{itemize}
\item Non-smooth regularisation. We are often faced with models of high complexity, however, the  solutions of interest often lie on a manifold of low dimension which is promoted by the non-smooth regulariser. There have been several recent studies explaining how proximal gradient methods identify this low dimensional manifold and efficiently output solutions which take a certain structure; see for instance \cite{liang2017activity} for the case of deterministic proximal gradient methods.
\item Stochastic methods. The past decades have seen an exponential growth in the data sizes that we have to handle, and stochastic methods have been popular due to their low computational cost; see for instance \cite{sag17,saga14,proxsvrg14} and references therein. 
\end{itemize}
The purpose of this paper is to show that proximal variance reduced stochastic gradient methods allow to benefit from both efficient structure enforcement and low computational cost. In particular, we present a study of manifold identification and local acceleration properties of these methods when applied to the following structured minimisation problem:
\beq\label{eq:min-Phi}\tag{$\mathcal{P}$}
\min_{x\in\bbR^n} ~~\Phi(x) \eqdef R(x) + F(x)    ,
\eeq
where $R(x)$ is a non-smooth structure imposing penalty term, and
\[
F(x) \eqdef \sfrac{1}{m} \msum_{i=1}^{m} f_{i}(x)  
\]
is the average of a finite sum, where each $f_{i}$ is smooth differentiable. 
We are interested in the problems where the value of $m$ is very large. 
A classic example of \eqref{eq:min-Phi} is $\ell_{1}$-norm regularised least square estimation (\ie the LASSO problem), which reads
\[
	\min_{x\in\bbR^n} \mu \norm{x}_{1} + \sfrac{1}{m} \msum_{i=1}^{m} \sfrac{1}{2} \norm{\calK_{i}x - b_{i}}^2 ,
\]
where $\mu>0$ is the trade-off parameter, $\calK_{i}$ is the $i^{\textrm{th}}$ row of a matrix $\calK \in \bbR^{m \times n}$, and $b_{i}$ is the $i^{\textrm{th}}$ element of the vector $b \in \bbR^{m}$. 
More examples of problem \eqref{eq:min-Phi} can be found in Section \ref{sec:experiment}.

Throughout this paper, we consider the following basic assumptions for problem \eqref{eq:min-Phi}:
\begin{enumerate}[leftmargin=4.5em,label= {\rm (\textbf{A.\arabic{*}}) }, ref= {\rm \textbf{A.\arabic{*}}}]
\item \label{A:R} 
$R : \bbR^n \to {\bbR} \cup \ba{+\infty}$ is proper, convex and lower semi-continuous;
\item \label{A:F} 
$F : \bbR^n \to {\bbR}$ is continuously differentiable with $\nabla F$ being $L_{F}$-Lipschitz continuous. For each index $i=1,\dotsm,m$, $f_{i}$ is continuously differentiable with $L_{i}$-Lipschitz continuous gradient;
\item \label{A:minimizers-nonempty} 
$\Argmin(\Phi) \neq \emptyset$, that is the set of minimisers is non-empty.
\end{enumerate}
In addition to assumption \iref{A:F}, define 
\[
L \eqdef \max_{i=\ba{1,\dotsm,m}} L_{i} ,
\]
which is the uniform Lipschitz continuity of functions $f_{i}$. 
Note that $L_F \leq \frac{1}{m}\sum_i L_i \leq L$ holds.

\subsection{Deterministic Forward--Backward splitting method}

A classical approach to solve \eqref{eq:min-Phi} is the Forward--Backward splitting (FBS) method \cite{lions1979splitting}, which is also known as the \emph{proximal gradient descent method}.
Given a current point $\xk$, the standard non-relaxed Forward--Backward iteration updates the next point $\xkp$ based on the following rule,
\beq \label{eq:fbs}
\xkp = \prox_{\gamma_k R} \Pa{\xk - \gamma_k \nabla F(\xk)} ,~~ \gamma_k \in ]0, 2/L_{F}[  ,
\eeq
where $\prox_{\gamma R}$ is the \emph{proximity operator} of $R$ which is defined as
\beq\label{eq:proximity-opt}
\prox_{\gamma R} (\cdot) \eqdef \min_{x\in\bbR^n} \gamma R(x)  +  \qfrac{1}{2}\norm{x-\cdot}^2  .
\eeq
Throughout this paper, unless otherwise stated, ``Forward--Backward splitting'' or ``FBS'' refers to the deterministic Forward--Backward splitting scheme \eqref{eq:fbs}. 

Since the original work \cite{lions1979splitting}, the properties of Forward--Backward splitting have been extensively studied in the literature. In general, the advantages of this method can be summarised as following:
\begin{itemize}
\item 
Robust convergence guarantees. The convergence of the method is guaranteed as long as $0 < \ugamma  \leq \gamma_{k} \leq \ogamma < 2/L_{F} $ holds for some $\ugamma, \ogamma > 0$ \cite{combettes2005signal}, for both the sequence $\seq{\xk}$ and the objective function value $\seq{\Phi(\xk)}$; 
\item
Known convergence rates. It is well established that the sequence of FBS scheme converges at the rate of $\norm{\xk-\xkm} = o(1/\sqrt{k})$ \cite{liang2014convergence}, while the objective function converges at the rate of $\Phi(\xk)-\Phi(\xsol) = o(1/k)$~\cite{molinari2018convergence} where $\xsol \in \Argmin(\Phi)$ is a global minimiser. These rates can be improved to linear\footnote{Linear convergence is also known as geometric or exponential convergence.} if for instance strong convexity is assumed \cite{nesterov2004introductory};
\item
Numerous acceleration techniques. Extensive acceleration schemes have been proposed over the decades, for instance the inertial schemes which contains inertial FBS \cite{moudafi2003convergence,lorenz2015inertial,liang2017activity}, FISTA \cite{fista2009} and Nesterov's optimal methods \cite{nesterov2004introductory};
\item 
Structure adaptivity. There has been several recent work \cite{liang2014local,liang2017activity} exploring the manifold identification properties of FBS, in particular, under the non-degeneracy condition that
\beq\label{eq:cnd-nd}\tag{ND}
- \nabla F(\xsol) \in \ri\Pa{\partial R(\xsol)}  ,
\eeq 
where $\ri\pa{\partial R(\xsol)}$ denotes the \emph{relative interior} of the sub-differential $\partial R(\xsol)$. 
It is shown in \cite{liang2017activity} that after a finite number of iterations, the FBS iterates $x_k$ all lie on the same manifold as the optimal solution $\xsol$. In the case of $R = \norm{\cdot}_1$, this equates to saying that there exists some $K\in\NN$ such that $x_k$ has the same sparse pattern as $\xsol$ for all $k\geq K$. Furthermore, upon identifying this optimal manifold, the FBS iterates can be proved to converge linearly to the optimal solution~$\xsol$. 
\end{itemize}
However, despite the above advantages of FBS, for the considered problem \eqref{eq:min-Phi}, when the value of $m$ is very large, the computational cost of $\nabla F(\xk)$ could be very expensive, which makes the deterministic FBS-type methods unsuitable for solving the large-scale problems arising from machine learning.

\subsection{Proximal variance reduced stochastic gradient methods}

The most straightforward extension of   stochastic gradient descent to the ``smooth + non-smooth'' setting is the \emph{proximal stochastic gradient descent} (Prox-SGD), which reads
\beq\label{eq:psgd}
\begin{aligned}
&\textrm{For $k = 0, 1,2,3,\dotsm$} \\
&\left\lfloor \begin{aligned}
& \textrm{sample $i_{k}$ uniformly from $\ba{1,\dotsm,m}$}  \\
& \xkp = \prox_{\gamma_{k} R} \Pa{ \xk - \gamma_{k} \nabla f_{i_{k}} (\xk) }  .
\end{aligned}
\right. \\
\end{aligned}
\eeq
The advantage of Prox-SGD over FBS scheme is that at each iteration, Prox-SGD only evaluates the gradient of one sampled function $f_{\ik}$, while FBS needs to compute $m$ gradients. 
However, to ensure the convergence of Prox-SGD, the step-size $\gammak$ of Prox-SGD has to converge to $0$ at a proper speed (\eg $\gammak = k^s$ for $s \in ]1/2, 1]$), leading to only $O(1/\sqrt{k})$ convergence rate for $\Phi(\xk)-\Phi(\xsol)$. 
Moreover, when $\Phi$ is strongly convex, the rate for the objective can only be improved to $O(1/k)$ which is much slower than the linear rate of FBS.

\paragraph{Prox-SGD has no manifold identification properties} 
Besides slow convergence speed, another disadvantage of Prox-SGD, when compared to FBS, is that the sequence $\seq{\xk}$ generated by the method  is unable to identify the structure of the problem, \ie no finite time manifold identification property. 

To give an intuitive explanation as to why the iterates of Prox-SGD are inherently unstructured, we first provide an alternative perspective of treating Prox-SGD, the perturbation of deterministic Forward--Backward splitting method. 
More precisely, this method can be written as the inexact Forward--Backward splitting method with stochastic approximation error on the gradient,
\beq\label{eq:inexact-fbs}
\begin{aligned}
&\textrm{For $k = 0, 1,2,3,\dotsm$} \\
&\left\lfloor \begin{aligned}
& \textrm{sample $\epsk$ from a finite distribution $\calD_{k}$} , \\
& \xkp = \prox_{\gamma_k R} \Pa{\xk - \gamma_k ( \nabla F(\xk) + \epsk) }  . \\
\end{aligned}
\right. \\
\end{aligned}
\eeq
For most stochastic gradient methods, we have $\EE\sp{\epsk} = 0$ and $\norm{\epsk}^2$ is the variance of the stochastic gradient. 
The stochastic approximation error $\epsk$ for Prox-SGD takes the form
\beq\label{eq:error-psgd}
\epsksgd \eqdef \nabla f_{i_{k}} (\xk) - \nabla F(\xk)  .
\eeq

Manifold identification for FBS can be guaranteed under the non-degeneracy condition \eqref{eq:cnd-nd}. In fact, from the definition of proximity operator \eqref{eq:proximity-opt}, at each iteration, we have 
\[
g_k \eqdef - \qfrac{x_{k+1}-x_{k}}{\gamma_k} -\nabla F(x_k) - \epsk \in \partial R(x_{k+1}) 
\]
and manifold identification can be guaranteed if $g_k \to -\nabla F(\xsol) $ as $ k\to\infty.$ The issue in the case of Prox-SGD is that although we have that in expectation $\EE\sp{\nabla f_{\ik}(\xk)} = \nabla F(\xk)$,  the error $\epsksgd$ is only bounded and in general does not converge to $0$ even if $\seq{\xk}$ converges to a global minimiser $\xsol \in \Argmin(\Phi)$. 

\vgap

We present a simple example to illustrate that Prox-SGD does not have manifold identification properties in general. Consider the following LASSO problem in 3D,
\[
\min_{x\in\bbR^3} \sfrac{1}{3} \norm{x}_{1} + \qfrac{1}{3} \msum_{i=1}^{3} \sfrac{1}{2} \norm{\calK_{i}x - b_{i}}^2   ,
\]
where
\[
\calK =
\begin{bmatrix}
1 & 0 & 0 \\
0 & \sqrt{2} & 0 \\
0 & 0 & \sqrt{3}
\end{bmatrix}
\qandq~~
b =
\begin{pmatrix}
2 \\ \sqrt{2}/3 \\ \sqrt{3}/4
\end{pmatrix}  .
\]
The optimal solution of this particular problem is $\xsol = (1,0,0)^T$ and writing $F(x)\eqdef \frac{1}{6}\norm{\calK x-b}^2$, we have that the non-degeneracy condition
\[
-\nabla F(\xsol) = \qfrac{1}{3} \begin{pmatrix}
1\\
{2}/3\\
{3}/4
\end{pmatrix} \in \ri\Ppa{\sfrac{1}{3} \partial \norm{\xsol}_1}  ,
~~~\textrm{where}~~~
(\partial \norm{\xsol})_{i}
= 
\sign\pa{x_{i}} = 
\left\{
\begin{aligned}
	&\enskip+1 &&: x_{i} > 0  ,  \\
	&[-1, +1] &&: x_{i} = 0  ,  \\
	&\enskip-1 &&: x_{i} < 0  ,
\end{aligned}
\right.
\]
It is furthermore straightforward to verify that $\norm{\nabla f_i(x) - \nabla F(x)}\geq \norm{\nabla F(x)}$ for all $i=1,2,3$. 
Moreover, if Prox-SGD is starting with $x_0 = (\mu,0,0)^T$ with $\mu\in\RR$, then with probability $2/3$ the first iterate of the algorithm satisfies $x_1\not\in \Mm_{\xsol} \eqdef \enscond{(x,0,0)}{x\in\RR}$. In fact, $x_1$ will have 2 non-zero entries if $\abs{\mu}>\gamma_1$ and $i_1 \in \{2,3\}$. Figure \ref{fig:sgd-noiden} shows the support sizes of the Prox-SGD iterates over $10^6$ iterations.
\begin{figure}[!ht]
{\centering
\includegraphics[width=0.5\linewidth]{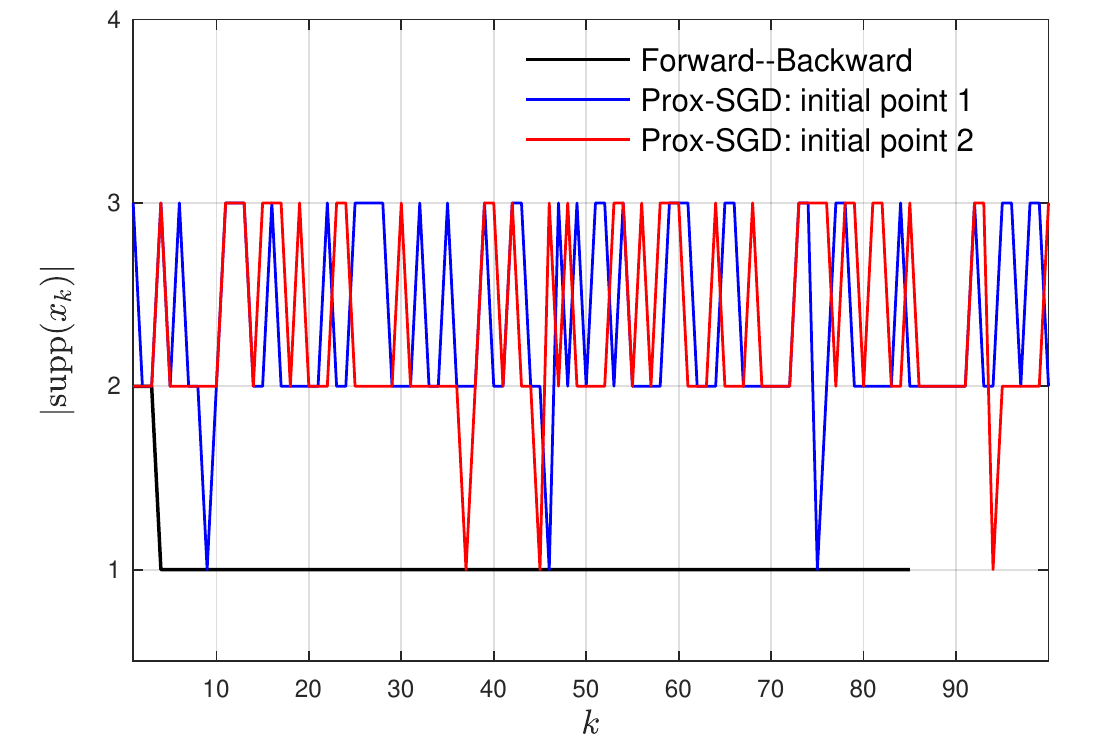}  
\caption{
Support identification comparison between FBS and Prox-SGD. For Prox-SGD, ``initial point 1'' starts with an arbitrary point with all three elements non-zero;
``initial point 2'' starts with the point $10\xsol$. 
The maximum number of iteration for Prox-SGD is $10^6$, the blue and red lines are sub-sampled, one out of every $10^4$ points.
}
\label{fig:sgd-noiden}
\vspace{-1em}}
\end{figure}

\subsubsection{Variance reduced methods}
To overcome the vanishing step-size and slow convergence speed of Prox-SGD, various (stochastic) incremental schemes are developed in literature; see for instance \cite{blatt2007convergent,vanli2016global,roux2012stochastic,sag17,saga14,svrg,proxsvrg14} and the references therein. 
Under stochastic setting, the variance reduced techniques are very popular approach, which have the following two main characteristics:
\begin{itemize}
\item 
Same as Prox-SGD, in expectation, the stochastic gradient remains an unbiased estimation of the full gradient;
\item
Different from Prox-SGD, the variance $\norm{\epsk}^2$ converges to $0$ when $\xk$ approaches the solution $\xsol$.
\end{itemize}
In the following, we introduce two well-known examples of variance reduced methods, the SAGA algorithm~\cite{saga14} and Prox-SVRG algorithm~\cite{proxsvrg14}, which are the main targets of this paper.

\paragraph{SAGA algorithm \cite{saga14}}

Similar to Prox-SGD algorithm, at each iteration $k$, the gradient of a sampled function $\nabla f_{i_{k}}(\xk)$ is computed by the SAGA algorithm where $i_{k}$ is uniformly sampled from $\ba{1,\dotsm, m}$. 
In the meantime, let $\ba{\nabla f_{i_{j}}(x_{k-j})}_{j=1,\dotsm,m}$ be the gradients history over the past $m$ steps, then the combination of these two aspects with additional debiasing yield the unbiased gradient approximation of the SAGA algorithm. 

Given an initial point $x_{0}$, define the individual gradient $g_{0,i} \eqdef  \nabla f_{i} (x_{0}), i=1,\dotsm,m$. Then
\beq\label{eq:saga}
\begin{aligned}
&\textrm{For $k = 0, 1,2,3,\dotsm$} \\
&\left\lfloor \begin{aligned}
& \textrm{sample $i_{k}$ uniformly from $\ba{1,\dotsm,m}$} , \\
& \xkp = \prox_{\gamma_{k} R} \Pa{ \xk - \gamma_{k} \pa{ \nabla f_{i_{k}} (\xk) - g_{i_{k}, k}  +  \sfrac{1}{m} \ssum_{i=1}^{m} g_{i, k}  } }  , \\
& \textrm{update the gradient history:}~~ g_{k,i} = \left\{ \begin{aligned}  &\nabla f_{i}(\xk) && \textrm{if~~} i = i_{k}  ,  \\ &g_{k-1,i} && \textrm{o.w.} \end{aligned} \right. 
\end{aligned}
\right. \\
\end{aligned}
\eeq
SAGA successfully avoids the vanishing step-size of Prox-SGD, and has the same convergence rate as Forward--Backward splitting scheme. 
However, one distinctive drawback of SAGA is that, in general, its memory cost is proportional to the number of functions $m$. 

In the context of \eqref{eq:inexact-fbs}, the stochastic approximation error $\epsk$ of SAGA takes the form
\beq\label{eq:error-saga}
\epsksaga \eqdef \nabla f_{i_{k}} (\xk) - g_{k, i_{k}}  +  \qfrac{1}{m} \msum_{i=1}^{m} g_{k, i} - \nabla F(\xk)  .
\eeq

\paragraph{Prox-SVRG algorithm}
The SVRG \cite{svrg} (stochastic variance reduced gradient) method was originally proposed to solve \eqref{eq:min-Phi} with $R = 0$, later on in \cite{proxsvrg14} it is extended to the case of $R$ being non-trivial. 
Compared to SAGA, in stead of approximating the current gradient $\nabla F$ with the past $m$ gradients $\nabla f_{\ik}$, Prox-SVRG computes the full gradient of a given point along the iteration, and uses it for $P$ iterations where $P$ is on the order of $m$. 

Let $P$ be a positive integer. The iteration of the algorithm consists of two level of loops, for the sequence $\txl$ in the outer loop, full gradient $\nabla F(\txk)$ is computed. For the sequence $x_{\ell, p}$ in the inner loop, only the gradient of the sampled function is computed. 
\beq\label{eq:prox-svrg}
\begin{aligned}
&\textrm{For $\ell = 0,1,2,3,\dotsm$} \\
&\left\lfloor 	\begin{aligned}
& \tilde{g}_{\ell} = \tfrac{1}{m} \ssum_{i=1}^{m} \nabla f_{i}(\txl) , x_{\ell, 0} = \tilde x_{\ell},  \\
&\textrm{For $p=1,\dotsm,P$} \\
&\left\lfloor 	\begin{aligned}
			& \textrm{sample $i_{p}$ uniformly from $\ba{1,\dotsm,m}$}  \\
			& x_{\ell, p} = \prox_{\gamma_{k} R} \Pa{ x_{\ell, p-1} - \gamma_{k} \pa{ \nabla f_{i_{p}}(x_{\ell, p-1}) - \nabla f_{i_{p}}(\txl) + \tilde{g}_{\ell}  } }  .
		\end{aligned}
\right. \\
& \text{Option I}: \txlp = x_{\ell, P}  , \\ 
& \text{Option II}:  \txlp = \tfrac{1}{P} \ssum_{p=1}^{P} x_{\ell, p}  .
\end{aligned}
\right. \\
\end{aligned}
\eeq
Prox-SVRG can also afford non-vanishing step-size and has the same convergence rate as FBS scheme. It avoids the large memory cost of SAGA, however, the gradient evaluation complexity of Prox-SVRG is always higher than SAGA. For instance when $P=m$, the gradient evaluation of Prox-SVRG is three times that of  SAGA. 

In the context of \eqref{eq:inexact-fbs}, given $x_{\ell, p}$, denote $k = \ell P + p$, then we have $x_{\ell, p} = \xk$ and the stochastic approximation error $\epsk$ of Prox-SVRG reads
\beq\label{eq:error-proxsvrg}
\begin{aligned}
\epsksvrg 
&= \nabla f_{i_{p}}(\xk) - \nabla f_{i_{p}}(\txl) + \tilde{g}_{\ell} - \nabla F(\xk)  .
\end{aligned}
\eeq

\subsection{Contributions}

In recent years, local linear convergence behaviours of the deterministic FBS-type methods have been studied under various scenarios. Particularly, in \cite{liang2017activity}, based on the notion of partial smoothness (see Definition~\ref{dfn:psf}), the authors propose a unified framework for local linear convergence analysis of Forward--Backward splitting and its variants including inertial FBS and FISTA \cite{fista2009,chambolle2015convergence}. 

In contrast to the deterministic setting, for the stochastic version of FBS scheme,  very limited results of this nature have been reported in the literature. 
However, in practice local linear convergence of stochastic proximal gradient descent has been observed without global strong convexity. 
More importantly, the low dimensional property of partial smoothness naturally reduces the computational cost and provides rich possibilities of acceleration. 
As a consequence, the lack of uniform analysis framework and exploiting the local acceleration are the main motivations of this work. 

\vspace{-0.6em}

\paragraph{Convergence of sequence for SAGA/Prox-SVRG}
Assuming only convexity, we prove the almost sure global convergence of the sequences generated by SAGA (see Theorem~\ref{thm:convergence-saga}) and Prox-SVRG with ``Option I'' (see Theorem~\ref{thm:convergence-proxsvrg}). 
Moreover, for Prox-SVRG algorithm with ``Option I'', an $O(1/k)$ ergodic convergence rate for the objective function is proved; see Theorem \ref{thm:convergence-proxsvrg}. 

\vspace{-0.6em}

\paragraph{Finite time manifold identification}
Let $\xsol \in \Argmin(\Phi)$ be a global minimiser of problem \eqref{eq:min-Phi}, and suppose that the sequence $\seq{\xk}$ generated by the  perturbed Forward--Backward \eqref{eq:inexact-fbs} converges to $\xsol$ almost surely. 
Then under the additional assumptions that the non-smooth function $R$ is partly smooth at $\xsol$ relative to a $C^2$-smooth manifold $\Msol$ (see Definition \ref{dfn:psf}) and a {non-degeneracy condition} (Eq.~\eqref{eq:cnd-nd}) holds at $\xsol$, in Theorem \ref{thm:abstract-identification} we prove a  general finite time manifold identification result for the perturbed Forward--Backward splitting scheme \eqref{eq:inexact-fbs}. 
The manifold identification means that after a finite number of iterations, say $K$, there holds $\xk \in \Msol$ for all $k \geq K$.

Specialising the result to SAGA and Prox-SVRG algorithms, we prove the finite manifold identification properties of them (see Corollary \ref{cor:identification}). 

\vspace{-0.6em}

\paragraph{Local linear convergence for SAGA/Prox-SVRG}
Building upon the manifold identification result, if moreover $F$ is locally $C^2$-smooth along $\Msol$ near $\xsol$ and a restricted injectivity condition (see Eq. \eqref{eq:cnd-ri}) is satisfied by the Hessian $\nabla^2 F(\xsol)$, we show that $\xsol$ is the unique minimiser of problem \eqref{eq:min-Phi} and moreover~$\Phi$ has local quadratic grow property around $\xsol$. 
As a consequence, we show that locally SAGA and Prox-SVRG converge linearly.

\vspace{-0.6em}

\paragraph{Local accelerations}
Another important implication of manifold identification is that the global non-smooth optimisation problem $\Phi$  becomes $C^2$-smooth locally along the manifold $\Msol$, and moreover is locally strongly convex if the restricted injectivity condition \eqref{eq:cnd-ri} is satisfied. 
This implies that locally we have many choices of acceleration to choose, for instance we can turn to higher-order optimisation methods, such as (quasi)-Newton methods or (stochastic) Riemannian manifold based optimisation methods which can lead to super linear convergence speed. 

\vgap

Lastly, for the numerical experiments considered in this paper, the corresponding MATLAB source code to reproduce the results is available online\footnote{\url{https://github.com/jliang993/Local-VRSGD}}.

\subsubsection{Relation to previous work}

Prior to our work, the identification properties of the \emph{regularised dual averaging algorithm} (RDA) \cite{xiao2010dual} were reported in \cite{lee2012manifold,duchi2016local}. 
The RDA algorithm is also proposed for solving problem \eqref{eq:min-Phi}, except that instead of being a finite sum, now the $F$ takes the form
\[
F(x) \eqdef \EE_{\xi} \Sp{ f(x; \xi) } = \int_{\Omega} f(x; \xi) \mathrm{d} \calD(\xi)  ,
\]
where $\xi \in \bbR^m$ is a random vector whose probability distribution $\calD$ is supported on the set $\Omega \subset \bbR^m$.

The RDA algorithm (\cite[Algorithm 1]{lee2012manifold}) for solving \eqref{eq:min-Phi} takes the following form, let $\xi_{1} = 0$ and $g_{0} = 0, \bar{g}_{0} =~0$,
\beq\label{eq:rda}
\begin{aligned}
&\textrm{For $k = 1,2,3,\dotsm$} \\
&\left\lfloor \begin{aligned}
& \textrm{sample $\xi_{k}$ from the distribution $\calD$, and compute: $\gk = \nabla f(\xk; \xi_{k})$;}  \\
& \textrm{update the averaged gradient: $\gbark = \sfrac{k-1}{k} \gbarkm + \sfrac{1}{k}\gk$;}  \\
& \textrm{update new point:}~ \xkp = \prox_{\tfrac{k}{\gamma_{k}} R} \Ppa{ - \sfrac{k}{\gamma_{k}} \gbark }  .
\end{aligned}
\right. \\
\end{aligned}
\eeq
Though proposed for \emph{infinite sum} problem, RDA can also applied to solve the \emph{finite sum} problem, moreover the convergence properties establish in \cite{lee2012manifold} remain hold. As a consequence, the identification property established there also holds true for the finite sum problem. 

Compare the proposed work and those of \cite{lee2012manifold,duchi2016local}, there are several differences need to be pointed out:
\begin{itemize}
\item
SAGA/Prox-SVRG and RDA are two very different types of methods. Although RDA can be applied to the finite sum problem, similarly to Prox-SGD, only $O(1/k)$ convergence rate can be achieved under strong convexity. While for the variance reduced SAGA/Prox-SVRG algorithms, linear convergence are available under strong convexity;
\item
For RDA algorithm, to the best of our knowledge, with only convexity assumption, so far there is no convergence result for the generated sequence $\seq{\xk}$. While for SAGA/Prox-SVRG algorithms, in this paper we prove the convergence properties of their generated sequences under only convexity assumption, which are new to the literature.
\end{itemize}

\subsection{Mathematical background}

Throughout the paper, $\bbN$ denotes the set of non-negative integers and $k \in \bbN$ denotes the index. 
$\bbR^n$ is the Euclidean space of $n$ dimension, and $\Id$ denotes the identity operator on $\bbR^n$. For a non-empty convex set $\Omega \subset \bbR^n$, $\ri(\Omega)$ and $\rbd(\Omega)$ denote its relative interior and boundary respectively, $\Aff(\Omega)$ is its affine hull, and $\LinHull(\Omega)$ is the subspace parallel to it. 
Denote $\proj_{\Omega}$ the orthogonal projector onto $\Omega$.

Given a proper, convex and lower semi-continuous function $R$, the sub-differential is defined by $\partial R(x) \eqdef \Ba{ g\in\bbR^{n} | R(y) \geq R(x) + \iprod{g}{y-x} ,~ \forall y\in\bbR^{n} }$. 
We say function $R$ is $\alpha$-strongly convex for some $\alpha > 0$ if $R(x) - \frac{\alpha}{2} \norm{x}^2$ still is convex.

\paragraph{Paper organisation}
The rest of the paper is organised as follows. 
In Section \ref{sec:global-convergence} we study the global convergence property of the sequence generated by SAGA and Prox-SVRG algorithms. The finite time manifold identification result is presented in Section \ref{sec:identification}. 
Local linear convergence and several local acceleration approaches are discussed in Section~\ref{sec:local-rate}. We conclude the paper with various numerical experiments in Section \ref{sec:experiment}. Several proofs of theorems are organised in Appendix \ref{appendix:proof}.

\section{Global convergence of SAGA/Prox-SVRG}
\label{sec:global-convergence}

In literature, though the global almost sure convergence of the objective function value  of SAGA/Prox-SVRG are well established \cite{saga14,proxsvrg14}, the convergence properties of the generated sequences are not proved unless strong convexity is assumed. In this section, we prove the almost sure convergence of the sequence generated by SAGA and Prox-SVRG with ``Option I'' without strong convexity assumption. The proofs of the theorems are provided in Appendix \ref{appendix:proof-sec2}.

\vspace{6pt}

We present first the convergence of the SAGA algorithm, recall that $L$ is the uniform Lipschitz continuity of all element functions $f_{i}, i=1,\dotsm,m$. 

\begin{theorem}[Convergence of SAGA]\label{thm:convergence-saga}
For problem \eqref{eq:min-Phi}, suppose that conditions \iref{A:R}-\iref{A:minimizers-nonempty} hold. Let $\seq{\xk}$ be the sequence generated by the SAGA algorithm \eqref{eq:saga} with $\gamma_k \equiv \gamma = 1/(3L)$, then there exists an $\xsol \in \Argmin(\Phi)$ such that almost surely we have $\Phi(\xk) \to \Phi(\xsol)$, $\xk \to\xsol$ and $\epsksaga \to 0$.
\end{theorem}

Next we provide the convergence result of the Prox-SVRG algorithm, and mainly focus on ``Option I'' for which convergence without strong convexity can be obtained. For ``Option II'',  convergence of the sequence under strong convexity is discussed already in \cite{proxsvrg14}, hence we decide to skip the discussion here.

Given $\ell \in \bbN^+$ and $p \in \ba{1,\dotsm,P}$, denote $k = \ell P + p$, then $x_{\ell, p} = \xk$. For sequence $\seq{\xk}$, define
\[
\xbark \eqdef \qfrac{1}{k} \msum_{\ell=1}^k x_{\ell}  .
\]

\begin{theorem}[Convergence of Prox-SVRG]\label{thm:convergence-proxsvrg}
For problem \eqref{eq:min-Phi}, suppose that conditions \iref{A:R}-\iref{A:minimizers-nonempty} hold. Let $\seq{\xk}$ be the sequence generated by the Prox-SVRG algorithm \eqref{eq:prox-svrg} with ``Option I''. 
Then,
\begin{enumerate}[label={\rm (\roman{*})}]
\item
If we fix $\gamma_k \equiv \gamma$ with $\gamma \leq \frac{1}{ 4L  (P+2)}$, then there exists a minimiser $\xsol  \in \Argmin(\Phi)$ such that  $\xk \to \xsol$ and $\epsksvrg \to 0$ almost surely. Moreover, {there holds for each $k = \ell P$ with $\ell\in\NN$,}
\begin{equation}\label{eq:svrg-rate}
\EE \Sp{ \Phi (\xbark) -\Phi(\xsol ) }  
\leq \qfrac{1}{k\gamma^2} \Pa{ \norm{\tilde x_0-\xsol }^2+(2\gamma-\gamma^2)(\Phi(\tilde x_0) - \Phi(\xsol )) }.
\end{equation}
\item
Suppose that $R, F$ are moreover $\alpha_{R}$ and $\alpha_{F}$ strongly convex respectively, then if $4L\gamma(P+1)< 1$, there holds
\[
\bbE\Sp{ \norm{\txl-\xsol}^2 }
\leq \rhosvrg^\ell \bPa{ \norm{\tilde{x}_{0} - \xsol}^2 + \sfrac{2\gamma}{1+\gamma\alpha_{R}} \Pa{ \Phi(\tilde{x}_{0}) - \Phi(\xsol) } }  ,
\]
where $\rhosvrg = \max\ba{ \frac{1-\gamma\alpha_{F}}{1+\gamma\alpha_{R}}, 4L\gamma(P+1) }$.
\end{enumerate}
\end{theorem}
\begin{remark}
To the best of our knowledge, the $O(1/k)$ ergodic convergence rate of $\seq{ \EE \sp{ \Phi (\xbark) -\Phi(\xsol ) } }$ is a new contribution to the literature.
\end{remark}


\section{Finite manifold identification of SAGA/Prox-SVRG}
\label{sec:identification}

From this section, we turn to the local convergence properties of SAGA/Prox-SVRG algorithms. We first introduce the notion of partial smoothness, then present a general abstract finite manifold identification of the perturbed Forward--Backward splitting \eqref{eq:inexact-fbs}, and specialize the result to the case of SAGA and Prox-SVRG algorithms.

\subsection{Partial smoothness}

The concept \emph{partial smoothness} was first proposed in \cite{Lewis-PartlySmooth}, which captures the essential features of the geometry of non-smoothness along the so-called active/identifiable manifold. Loosely speaking, a partly smooth function behaves smoothly along the manifold, and sharply normal to the manifold.

Let $\Mx$ be a $C^2$-smooth Riemannian manifold of $\bbR^n$ around a point $x$. 
Denotes $\tanSp{\Mx}{x'}$ the tangent space of $\Mx$ at a point $x' \in \Mx$. 
Below we introduce the definition of partial smoothness for the class of proper convex and lower semi-continuous functions. 

\begin{definition}[Partly smooth function]\label{dfn:psf}
Let function $R: \bbR^n \to \bbR \cup \ba{+\infty}$ be proper convex and lower semi-continuous. 
Then $R$ is said to be {partly smooth at $x$ relative to a set $\Mx$} containing $x$ if the sub-differential $\partial R(x) \neq \emptyset$, and moreover
\begin{description}[leftmargin=3cm]
\item[{\textbf{Smoothness:}}] \label{PS:C2} 
$\Mx$ is a $C^2$-manifold around $x$, $R$ restricted to $\Mx$ is $C^2$ around $x$.
\item[{\textbf{Sharpness:}}] \label{PS:Sharp} 
The tangent space $\tanSp{\Mx}{x}$ coincides with $T_{x} \eqdef \LinHull\Pa{\partial R(x)}^\perp$.
\item[{\textbf{Continuity:}}] \label{PS:DiffCont} 
The set-valued mapping $\partial R$ is continuous at $x$ relative to $\Mx$.
\end{description}
\end{definition}

The class of partly smooth functions at $x$ relative to $\Mx$ is denoted as $\PSF{x}{\Mx}$. 
Many widely used non-smooth penalty functions in the literature are partly smooth, such as sparsity promoting $\ell_{1}$-norm, group sparsity promoting $\ell_{1,2}$-norm, low rank promoting nuclear norm, etc.; see Table \ref{tab:psf-example} for more information. 
We refer to \cite{liang2017activity} and the references therein for more details of partly smooth functions.

\begin{table}[!htb]
\centering \small \caption{Examples of partly smooth functions. For $x \in \bbR^n$ and some subset of indices $\pzcb \subset \ens{1,\ldots,n}$, $x_{\pzcb}$ is the restriction of $x$ to the entries indexed in $\pzcb$. $\TVop$ stands for the finite differences operator.}
\begin{TAB}(r,0.45cm,0.45cm)[2pt]{|c|c|c|}{|c|c|c|c|c|c|}
Function  & Expression & Partial smooth manifold  \\
$\ell_1$-norm & $\norm{x}_1=\sum_{i=1}^n \abs{x_i}$ & $\Mx = T_x =\enscond{z \in \RR^n}{\calI_z \subseteq \calI_x},~ \calI_x = \enscond{i}{x_{i} \neq 0 }$   \\
$\ell_{1,2}$-norm & $\sum_{i=1}^m \norm{x_{\pzcb_i}}$ & $\Mx = T_x = \enscond{z \in \RR^n}{\calI_z \subseteq \calI_x} ,~ \calI_x = \enscond{i}{x_{\pzcb_i} \neq 0 }$~  \\
$\ell_\infty$-norm & $\max_{i=\{1,\ldots,n\}}\abs{x_i}$& $\Mx = T_x =\enscond{z \in \RR^n}{z_{\calI_x} \in \bbR \sign(x_{\calI_x})},~ \calI_x = \enscond{i}{\abs{x_i} = \norm{x}_\infty}$  \\
TV semi-norm & $\TVnorm{x}=\norm{\TVop x}_1$ & $\Mx = T_x = \enscond{z \in \bbR^n}{\calI_{\TVop z} \subseteq \calI_{\TVop x}},~ \calI_{\TVop x} = \ba{i: \pa{\TVop x}_{i} \neq 0 }$  \\
Nuclear norm & $\norm{x}_*=\sum_{i=1}^r \sigma(x)$ & $~~\Mx = \enscond{z \in \bbR^{n_1 \times n_2}}{\rank(z) = \rank(x) = r}$, $\sigma(x)$ singular values of $x$~~  \\
\end{TAB}
\label{tab:psf-example}
\end{table}

\subsection{An abstract finite manifold identification}

Recall the perturbed Forward--Backward splitting iteration
\beqn
\xkp = \prox_{\gamma_k R} \Pa{\xk - \gamma_k ( \nabla F(\xk) + \epsk) }  .
\eeqn
As discussed, the difference of stochastic optimisation methods in terms of perturbed Forward--Backward splitting is that each method has its own form of the perturbation error $\epsk$ (\eg $\epsksgd$ in \eqref{eq:error-psgd}, $\epsksaga$ in \eqref{eq:error-saga} and $\epsksvrg$ in~\eqref{eq:error-proxsvrg}).  
We have the following abstract identification result for the perturbed Forward--Backward iteration.

\begin{theorem}[Abstract finite manifold identification]\label{thm:abstract-identification}
For problem \eqref{eq:min-Phi}, suppose that conditions \iref{A:R}-\iref{A:minimizers-nonempty} hold. 
For the perturbed Forward--Backward splitting iteration \eqref{eq:inexact-fbs}, suppose that:
\begin{enumerate}[leftmargin=4.5em,label= {\rm (\textbf{B.\arabic{*}})},ref= {\rm \textbf{B.\arabic{*}}}]
\item \label{B:gamma-geq-eps}
There exists $\ugamma > 0$ such that $\liminf_{k\to+\infty} \gamma_{k} \geq \ugamma$;
\item \label{B:epsk-converge}
The perturbation error $\seq{\epsk}$ converges to $0$ almost surely;
\item \label{B:xk-converge}
There exists an $\xsol \in \Argmin(\Phi)$ such that $\seq{\xk}$ converges to $\xsol$ almost surely.
\end{enumerate}
For the $\xsol$ in \iref{B:xk-converge}, suppose that $R \in \PSF{\xsol}{\Msol}$, and the following non-degeneracy condition \eqref{eq:cnd-nd} 
holds. 
Then, there exists a $K>0$ such that for all $k\geq K$, we have $\xk \in \Msol$ almost surely.
\end{theorem}
\begin{remark}\label{rmk:remark-identification}
     $~$  
\begin{enumerate}[label={\rm (\roman{*})}]
\item
In the deterministic setting, the finite manifold identification property of \eqref{eq:inexact-fbs}, \ie $\epsk$ is not random error, is discussed in \cite[Section~3.3]{liang2017activity}.  
\item
From the convergence proof of Theorem \ref{thm:abstract-identification}, it can be observed that condition \iref{B:gamma-geq-eps} can be relaxed if we have 
\[
\lim_{k\to+\infty} \sfrac{1}{\gammak} \norm{\xk-\xkp} = 0 
\]
holds almost surely, which means that $\EE\sp{ \norm{\xk-\xkp} } = o(\gamma_{k})$. 
\item
In Theorem \ref{thm:abstract-identification}, we only mention the existence of $K$ after which the manifold identification happens and no estimation is provided. 
In \cite{liang2017activity} for the deterministic Forward--Backward splitting method, a lower bound of $K$ is derived, though not very interesting from practical point of view. 
However, for the stochastic methods (\eg SAGA and Prox-SVRG), even providing a lower bound for $K$ is a challenging problem. 
More importantly, to provide a bound (either lower or upper) for $K$, $\xsol$ has to be involved; see \cite[Proposition 3.6]{liang2017activity}.  As a consequence, we decide to skip the discussion here.
\end{enumerate}
\end{remark}
\begin{proof}

First of all, the definition of proximity operator \eqref{eq:proximity-opt} and the update of $\xkp$ \eqref{eq:inexact-fbs} entail that  
\beq\label{eq:opt-cnd-xkp}
\qfrac{\xk - \xkp}{\gammak} - \nabla F(\xk) - \epsk \in \partial R (\xkp) ,
\eeq
from which we get
\[
\begin{aligned}
\dist\Pa{-\nabla F(\xsol), \partial R(\xkp)}  
&\leq \norm{\sfrac{1}{\gammak}\Pa{\xk - \xkp} - \nabla F(\xk) - \epsk + \nabla F(\xsol)} \\
&\leq \sfrac{1}{\gammak} \norm{\xk-\xkp} + \norm{\nabla F(\xk) - \nabla F(\xsol)} + \norm{\epsk} \\
&\leq \sfrac{1}{\ugamma} \norm{\xkp-\xk} + L_{F} \norm{\xk - \xsol} + \norm{\epsk} ,
\end{aligned}
\]
where lower boundedness of $\gammak$ and the $L_{F}$-Lipschitz continuity of $\nabla F$ (see assumption \iref{A:F}) is applied to get the last inequality. 
We have:
\begin{itemize}
\item 
The almost sure convergence of $\seq{\xk}$ (condition \iref{B:xk-converge}) ensures that $L_{F} \norm{\xk - \xsol}$ converges to $0$ almost surely. Owing to assumption \iref{A:R}, $R$ is sub-differentially continuous at all the points of its domain, typically at $\xsol$ for $-\nabla F(\xsol)$, hence we have $R(\xk) \to R(\xsol)$ almost surely;
\item
Combine the almost sure convergence of $\seq{\xk}$ and \iref{B:gamma-geq-eps} the bounded from below property of $\seq{\gammak}$, we have that $\sfrac{1}{\ugamma} \norm{\xkp-\xk}$ converges to $0$ almost surely;
\item
Condition \iref{B:epsk-converge} asserts that $\norm{\epsk} \to 0$ almost surely.
\end{itemize}
Altogether, we have that
\[
    \dist\Pa{-\nabla F(\xsol), \partial R(\xkp)}  \to 0 ~~~\textrm{almost~surely}  .
\]
To this end, all the conditions of \cite[Theorem~5.3]{hare2004identifying} are fulfilled almost surely on function $\iprod{\nabla F(\xsol)}{\cdot}+R$, hence the identification result follows.
\end{proof}

\subsection{Finite manifold identification of SAGA/Prox-SVRG}

Now we specialise Theorem \ref{thm:abstract-identification} to the case of SAGA/Prox-SVRG algorithms, which yields the proposition below.
For Prox-SVRG, recall that in the convergence proof, we denote the inner iteration sequence $x_{\ell, p}$ as $\xk$ with $k = \ell P + p$. It follows directly from Theorem \ref{thm:convergence-saga} and Theorem \ref{thm:convergence-proxsvrg} that the conditions of Theorem \ref{thm:abstract-identification} are satisfied. Therefore, we have the following result.

\begin{corollary}\label{cor:identification}
For problem \eqref{eq:min-Phi}, suppose that conditions \iref{A:R}-\iref{A:minimizers-nonempty} hold. Suppose that
\begin{itemize}
\item 
SAGA is applied under the conditions of Theorem \ref{thm:convergence-saga}; 
\item
Prox-SVRG is applied under the conditions of Theorem \ref{thm:convergence-proxsvrg}. 
\end{itemize}
Then there exists an $\xsol \in \Argmin\pa{\Phi}$ such that the sequence $\seq{\xk}$ generated by either algorithm converges to $\xsol$ almost surely. 

If moreover, $R \in \PSF{\xsol}{\Msol}$, and the non-degeneracy condition \eqref{eq:cnd-nd} holds. 
Then, there exists a $K>0$ such that for all $k\geq K$, $\xk \in \Msol$ almost surely.
\end{corollary}
\begin{remark}
For the Prox-SVRG algorithm, since in Theorem \ref{thm:convergence-proxsvrg} the convergence is obtained for ``Option I'', hence the sequence $\ba{\txl}_{\ell\in\bbN}$ also has finite manifold identification property. 

{\parindent12pt The situation however becomes complicated if ``Option II'' is applied. Suppose we have the convergence of the sequence generated by Prox-SVRG, the identification property of $\ba{x_{\ell, p}}_{p=1,\dotsm,P,~ \ell\in\bbN}$ is straightforward. 
However, for sequence $\ba{\txl}_{\ell\in\bbN}$, unless $\Msol$ is convex locally around $\xsol$, in general there is no identification guarantee for it. 
A typical example for which this is problematic is the nuclear norm, whose associated partial smooth manifold is a non-convex cone, hence there can be no identification result for the outer loop sequence~$\ba{\txl}_{\ell\in\bbN}$. } 
\end{remark}

\subsection{When non-degeneracy condition fails}

In Theorem \ref{thm:abstract-identification}, besides the partial smoothness assumption of $R$, the non-degeneracy condition \eqref{eq:cnd-nd} is crucial to the identification of the sequence $\seq{\xk}$. 
Owing to the result of \cite{LewisPartlyTiltHessian,hare2004identifying,Hare-Lewis-Algo}, it is a necessary condition for identification of the manifold $\Msol$, and moreover ensures that the manifold $\Msol$ is minimal and unique. 

Recently, efforts are made to relax the non-degeneracy condition. 
In \cite{fadili2017sensitivity}, under a so-called ``mirror stratification condition'', the authors manage to relax the non-degeneracy condition, however at the price that the manifold to be identified is no longer unique. 
More precisely, there will be another manifold $\overline{\calM}_{\xsol}$, which includes $\Msol$ and is determined by how \eqref{eq:cnd-nd} is violated. The sequence $\seq{\xk}$ will identify a manifold $\widetilde{\calM}_{\xsol}$ such that
\[
\Msol \subseteq \widetilde{\calM}_{\xsol} \subseteq \overline{\calM}_{\xsol}  .
\]
Furthermore, the identification of $\seq{\xk}$ could be unstable, that is $\seq{\xk}$ may identify several different manifolds which are between $\Msol$ and $\overline{\calM}_{\xsol}$.

\paragraph{A degenerate LASSO problem}
We present a simple example of LASSO problem to demonstrate the unstable identification behaviour of $\seq{\xk}$ when the non-degeneracy conditions fails. 
Consider the problem
\beq\label{eq:lasso-dctmtx}
\min_{x\in\bbR^n}  \mu \norm{x}_{1} + \qfrac{1}{2} \norm{\calK x - b}^2 ,
\eeq
where $\mu > 0$ is the penalty parameter, $\calK \in \bbR^{n\times n}$ is a unitary matrix, and $b \in \bbR^n$ is a vector.

\begin{figure}[!ht]
\centering
\subfloat[$\xsol$ and $(\calK^T b - \xsol)/\mu$]{ \includegraphics[width=0.45\linewidth]{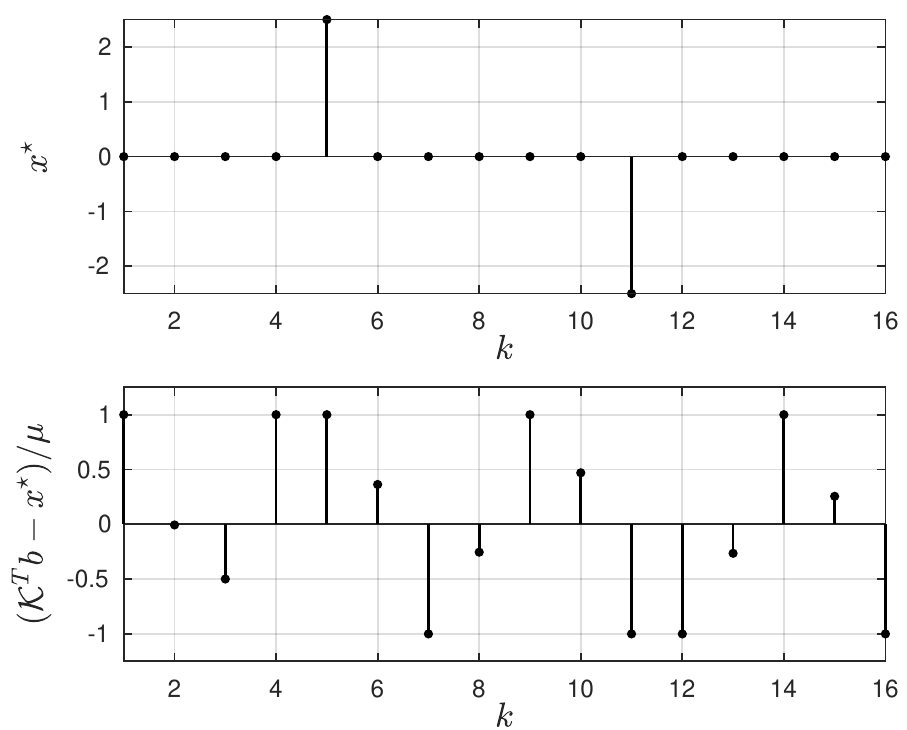} }  \hspace{2pt}
\subfloat[$\abs{\supp(\xk)}$ under different starting point]{ \includegraphics[width=0.45\linewidth]{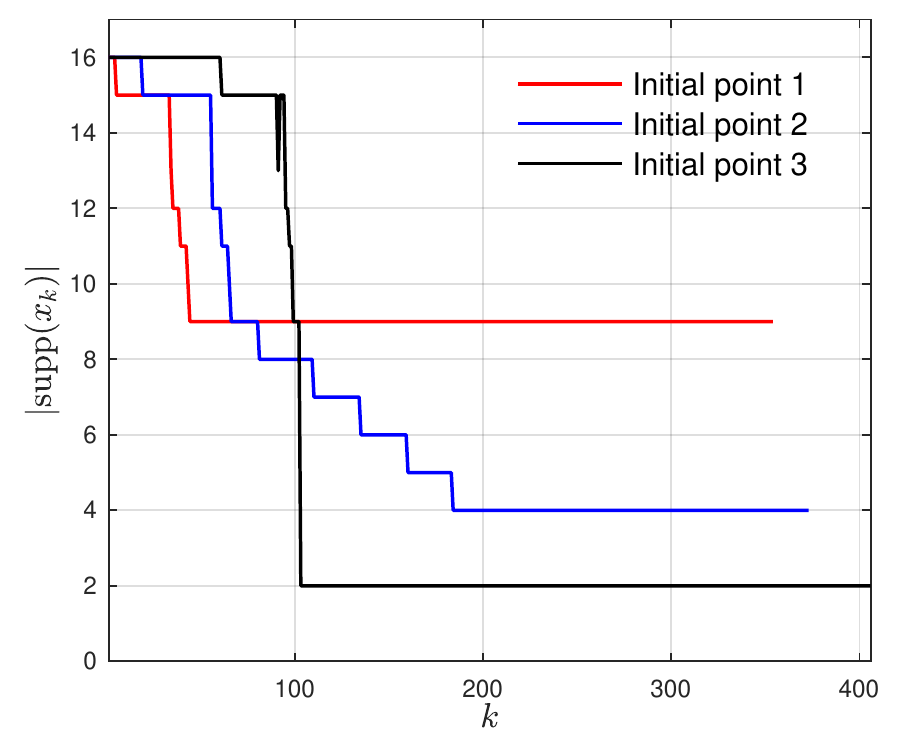} }  
\caption{Identification properties of deterministic Forward--Backward splitting method when the non-degeneracy condition \eqref{eq:cnd-nd} fails.
}
\label{fig:ndc-fails}
\end{figure}

Since $\calK$ is a unitary matrix, the solution of \eqref{eq:lasso-dctmtx} is unique and can be given explicitly, which is 
\beq\label{eq:solution-xsol}
\xsol = \sign(\calK^T b) \odot \max\Ba{ \abs{\calK^T b}-\mu, 0 } ,  
\eeq
and $\odot$ denotes point-wise product. Moreover, we have the gradient at $\xsol$
\[
- \nabla \Ppa{ \qfrac{1}{2} \norm{\calK \xsol - b}^2 } = -\calK^T (\calK \xsol - b) = \calK^T b - \xsol  .
\]
In the experiments, we set $\mu = 0.5$ and $n = 16$, and moreover the vector $b$ is designed such that the non-degeneracy condition~\eqref{eq:cnd-nd} is violated. 
The two vectors $\xsol$ and $\calK^T b - \xsol$ are shown in Figure \ref{fig:ndc-fails}(a), and it can be observed that $\xsol$ has only \emph{two} non-zero elements, while $\calK^T b - \xsol$ has \emph{nine} saturated elements (the saturation means that the absolute value of corresponding element is equal to $\mu$).

Though the solution $\xsol$ can be provided in closed form \eqref{eq:solution-xsol}, we choose to solve \eqref{eq:lasso-dctmtx} with deterministic Forward--Backward splitting with fixed step-size $\gamma = 0.05$, which is the following iteration
\beq\label{eq:ista}
\xkp = \sign(\wk) \odot \max\Ba{ \abs{\wk}-\gamma\mu, 0 }  ~~~~\textrm{where}~~~~ \wk = { (1 - \gamma) \xk -  \calK^T b }  .
\eeq
Three different initial points for \eqref{eq:ista} are considered. 
For each starting point, the size of support of the sequence $\seq{\xk}$, \ie $\seq{\abs{\supp\pa{\xk}}}$, is plotted in Figure \ref{fig:ndc-fails}(b). For all three cases, the iterations are ran until machine accuracy is reached. We obtain the following observations from the comparisons:
\begin{itemize}
\item
``Initial point 1'' and ``Initial point 2'' are unable to identify the support of the solution $\xsol$; 
\item
``Initial point 1'' identifies the largest manifold, \ie $\overline{\calM}_{\xsol}$. For ``Initial point 2'', the identification is not stable in the early iterations (\eg $k \leq 190$) compared to the other cases, and eventually (\eg $k \geq 190$) stabilises onto a manifold $\widetilde{\calM}_{\xsol}$ with $\Msol \subset \widetilde{\calM}_{\xsol} \subset \overline{\calM}_{\xsol}$;
\item
``Initial point 3'' manages to identify the smallest manifold, \ie $\Msol$. 
\end{itemize}
We can conclude that the starting point is very crucial when the non-degeneracy condition \eqref{eq:cnd-nd} fails.

\section{Local linear convergence of SAGA/Prox-SVRG} 
\label{sec:local-rate}

Now we turn to the  local linear convergence properties of SAGA/Prox-SVRG algorithms, the contents of this section consist of three main parts: local linear convergence of SAGA/Prox-SVRG, tightness of the rate estimation and more importantly acceleration techniques for these methods. 

Throughout the section, $\xsol \in \Argmin(\Phi)$ denotes a global minimiser \eqref{eq:min-Phi}, $\Msol$ is a $C^2$-smooth manifold which contains $\xsol$, and $\Tsol$ denotes the tangent space of $\Msol$ at $\xsol$.

\subsection{Local linear convergence}

Similar to the result in \cite{liang2017activity} for the deterministic FBS-type methods, the key assumption to establish local linear convergence for SAGA/Prox-SVRG is a so-called {restricted injectivity} condition defined below.

\paragraph*{Restricted injectivity}
Let $F$ be locally $C^2$-smooth around the minimiser $\xsol$, and moreover the following {restricted injectivity} condition holds
\beq\label{eq:cnd-ri}\tag{RI}
\ker\Pa{\nabla^2 F(\xsol)} \cap T_{\xsol} = \ba{0}  .
\eeq

Owing to the local continuity of the Hessian of $F$, condition \eqref{eq:cnd-ri} implies that there exist $\alpha > 0$ and $r > 0$ such that
\beqn
\iprod{h}{\nabla^2 F(x) h} \geq \alpha \norm{h}^2,~~ \forall h \in \Tsol  ,~ \forall x~~~\st~~\norm{x-\xsol} \leq r  .
\eeqn
In \cite[Proposition~12]{liang2017activity}, it is shown that under the above condition, $\xsol$ actually is the unique minimiser of problem~\eqref{eq:min-Phi}, and $\Phi$ grows locally quadratic if moreover $R \in \PSF{\xsol}{\Msol}$.

\begin{lemma}[Local quadratic growth \cite{liang2017activity}]\label{lem:quadratic-growth}
For problem \eqref{eq:min-Phi}, suppose that assumptions \iref{A:R}-\iref{A:minimizers-nonempty} hold. Let $\xsol \in \Argmin(\Phi)$ be a global minimiser such that conditions \eqref{eq:cnd-nd} and \eqref{eq:cnd-ri} are fulfilled and $R \in \PSF{\xsol}{\Msol}$, then $\xsol$ is the unique minimiser of \eqref{eq:min-Phi} and there exist $\alpha > 0$ and $r > 0$ such that
\[
\Phi(x) - \Phi(\xsol) 
\geq \alpha \norm{x - \xsol}^2 : \forall x ~~\st~~ \norm{x-\xsol} \leq r .
\]
\end{lemma}

\begin{remark}
A similar result can also be found in \cite[Theorem 5]{lee2012manifold}. 
\end{remark}

The local quadratic growth, implies that when a sequence convergent stochastic method is applied, and moreover the conditions of Lemma \ref{lem:quadratic-growth} are satisfied. 
Eventually, the method will enter a local neighbourhood of the solution $\xsol$ where the function has the quadratic growth property. If moreover the method is linearly convergent under strong convexity, then locally it will also converge linearly under quadratic growth. 
As a consequence, we have the following propositions for SAGA and Prox-SVRG respectively.

\begin{proposition}[Local linear convergence of SAGA]\label{prop:linear-rate-saga}
For problem \eqref{eq:min-Phi}, suppose that conditions \iref{A:R}-\iref{A:minimizers-nonempty} hold, and the SAGA algorithm \eqref{eq:saga} is applied with $\gamma_k \equiv \gamma = 1/(3L)$. Then $\xk$ converges to $\xsol \in \Argmin(\Phi)$ almost surely. 
If moreover, $R \in \PSF{\xsol}{\Msol}$, and conditions \eqref{eq:cnd-nd}-\eqref{eq:cnd-ri} are satisfied. Then there exists $K >0$ such that for all $k \geq K$, 
\beqn
\bbE\Sp{\norm{\xk - \xsol}^2} =  O( \rhosaga^{k-K} )  ,
\eeqn
where $\rhosaga = 1 - \min\ba{ \frac{1}{4m}, \frac{\alpha}{3L} }$.
\end{proposition}

We refer to \cite{saga14} for the proof of the proposition. 

\begin{remark}
Follow the result of SAGA paper, if locally we change to $\gamma = 1/(2(\alpha m + L))$, then we have for $\rhosaga$
\[
\rhosaga = { 1 - \alpha\gamma } = { 1 - \sfrac{\alpha}{2(\alpha m + L)} }  .	
\]
It also should be noted that $\gamma = \frac{1}{3L}$ is the optimal step-size for SAGA as pointed out in \cite{saga14}. 
\end{remark}

\begin{proposition}[Local linear convergence of Prox-SVRG]\label{prop:linear-rate-proxsvrg}
For problem \eqref{eq:min-Phi}, suppose that conditions \iref{A:R}-\iref{A:minimizers-nonempty} hold, and the Prox-SVRG algorithm \eqref{eq:prox-svrg} is applied such that Theorem \ref{thm:convergence-proxsvrg} holds. Then $\xk$ converges to $\xsol \in \Argmin(\Phi)$ almost surely. 
If moreover, $R \in \PSF{\xsol}{\Msol}$, and conditions \eqref{eq:cnd-nd}-\eqref{eq:cnd-ri} are satisfied. Then there exists $K >0$ such that for all $k \geq K$, 
\beqn
\bbE\Sp{\norm{\txl - \xsol}^2} =  O( \rhosvrg^{\ell-K} )  ,
\eeqn
where $\rhosvrg = \max\ba{ \frac{1-\gamma\alpha_{F}}{1+\gamma\alpha_{R}}, 4L\gamma(P+1) }$ and $\gamma,P$ are chosen such that $\rhosvrg < 1$.
\end{proposition}

The claim is a direct consequence of Theorem \ref{thm:convergence-proxsvrg}(ii). 

\begin{remark}
	$~$
\begin{enumerate}[label={\rm (\roman{*})}]
\item
When $P$ is large enough, then $\rhosvrg \approx \frac{1}{\alpha\gamma(1-4L\gamma)P} + \frac{4L\gamma}{1-4L\gamma}$, to make it strictly smaller than $1$, we need $P \geq 32L/\alpha$ and moreover 
\[
\gamma \in \bSp{ \sfrac{P\alpha - \sqrt{\Delta}}{16LP\alpha}, \sfrac{P\alpha + \sqrt{\Delta}}{16LP\alpha} }
~~~~\textrm{where}~~~
\Delta = P\alpha (P\alpha - 32L) .
\]
\item
In \cite{gong2014linear}, the authors studied the linear convergence convergence of Prox-SVRG under a ``semi-strongly convex'' assumption. Our assumption for local linear convergence is very close to this one, however in stead of only allowing polyhedral functions (typically $\ell_{1}$-norm), our analysis goes much further, for instance our result allows to analyse nuclear norm. 
\item
The above local linear convergence result is quite different from that of \cite[Section 4]{liang2017activity} for the deterministic FBS-type methods, which can be summarised into the following steps:
\begin{enumerate}[leftmargin=5.5em,label= {\rm \textbf{Step \arabic{*}.\hspace*{-4pt}} }, ref= {\rm \textbf{Step \arabic{*}}}]
\item \label{step1} 
Locally along the identified $\Msol$, the globally non-linear iteration \eqref{eq:fbs} can be linearised, resulting in a linear matrix $\mFB$;
\item \label{step2} 
Spectral properties of $\mFB$, conditions such that the spectral radius $\rho(\mFB) <~1$;
\item \label{step3} 
Local linear convergence of FBS-type splitting schemes.
\end{enumerate}
The advantage of this strategy is that it exploits explicitly the geometry of the manifold $\Msol$ and encodes it into the matrix $\mFB$, which result in a very tight rate estimation. 

{\parindent12pt The main difficulty of applying the above strategy to SAGA/Prox-SVRG is that, under the stochastic setting, the error $\epsk$ in \eqref{eq:inexact-fbs} cannot be controlled explicitly, which makes it impossible to use the spectral radius $\rho(\mFB)$ as rate estimation; see the section below for more details. }

\end{enumerate}
\end{remark}

\subsection{Better local rate estimation?}

Consider FBS and SAGA algorithms, when $\Phi$ is $\alpha$-strongly convex and the step-size is chosen as $\gamma = 1/(3L)$, then the convergence rate of $\norm{\xk-\xsol}$ for these two algorithms are
\[
\rhofb = {1 - \sfrac{\alpha}{3L}}	,~~~
\rhosaga = \sqrt{ 1 - \min\ba{ \tfrac{1}{4m}, \tfrac{\alpha}{3L} } }  .
\]
Clearly, the rate estimation of FBS is better than that of SAGA. Note that here we are comparing the convergence rate \emph{per iteration}, not based on gradient evaluation complexity. 
For the rest of this part, we will discuss the difficulties of improving the rate estimations for SAGA and Prox-SVRG.

\subsubsection{Local linearised iteration}

Follow the setting of \cite{liang2017activity}, suppose that $F$ locally around $\xsol$ is $C^2$-smooth, define the following matrices which are all \emph{symmetric}:
\beq\label{eq:mtxs}
\HF \eqdef \PT{T_{\xsol}} \nabla^2 F(\xsol) \PT{T_{\xsol}}  ,\quad
\GF \eqdef \Id - \gamma \HF ,\quad
\HR \eqdef \nabla^2_{\Msol} {\Phi}(\xsol)\PT{T_{\xsol}} -  \HF  ,
\eeq
where $\nabla^2_{\Msol} {\Phi}$ is the Riemannian Hessian of $\Phi$ along the manifold $\Msol$; see Lemma~\ref{lem:gradhess}.

\begin{lemma}[{\cite[Lemma~13,14]{liang2017activity}}]\label{lem:eigmtxs} 
For problem \eqref{eq:min-Phi}, suppose that conditions \iref{A:R}-\iref{A:minimizers-nonempty} hold and $\xsol \in \Argmin\pa{\Phi}$ such that $R \in \PSF{\xsol}{\Msol}$, $F$ is locally $C^2$ around $\xsol$ and conditions \eqref{eq:cnd-nd} and \eqref{eq:cnd-ri}~hold. 
\begin{enumerate}[label=\rm(\roman{*})]
\item 
$\HR$ is symmetric positive semi-definite, hence $\Id + \gamma \HR$ is invertible, and $\WR \eqdef \pa{\Id + \gamma \HR}^{-1}$ is symmetric positive definite with eigenvalues in $]0,1]$.
\item
Define the matrix $\mFB$ by
\beq\label{eq:mtx-M}
\mFB \eqdef  \WR \GF  .
\eeq
For $\gamma \in ]0, 1/L[$, $\mFB$ has real eigenvalues lying in $]0, 1[$ with spectral radius $\rho(\mFB) \leq 1 - \alpha \gamma$.
\end{enumerate}
\end{lemma}

\begin{proposition}[Local linearised iteration]\label{prop:linearisation}
For problem \eqref{eq:min-Phi}, suppose that conditions \iref{A:R}-\iref{A:minimizers-nonempty} hold. 
Assume the perturbed Forward--Backward iteration \eqref{eq:inexact-fbs} is applied to create a sequence $\seq{\xk}$ such that the conditions of Theorem~\ref{thm:abstract-identification} hold. 
Then there exists an $\xsol \in \Argmin\pa{\Phi}$ such that $\xk \to \xsol$ almost surely and $\xk \in \Msol$ for all $k$ large enough. 

If moreover, $F$ is locally $C^2$-smooth around $\xsol$ and $\gammak \to \gamma \in ]0, 1/L[$, then with probability one, there exists $K\in\NN$ such that for all $k\geq K$, we have 
\beq\label{eq:linearised-inexact-fb}
\dkp = \mFB \dk + \phik  ,
\eeq
where $\dk \eqdef \xk-\xsol + o(\norm{\xk-\xsol})$ and $\phik = \gamma \WR \PT{\Tsol} \epsk + o(\norm{\epsk})$.
\end{proposition}

See Appendix \ref{appendix:proof-sec4} for the proof. Note that for $\phik$, there still holds $\EE\sp{\phik} = 0$. 

\begin{remark}
In \cite{liang2017activity}, the linearisation of deterministic FBS scheme reads,
\[
\xkp - \xsol = \mFB(\xk - \xsol) + o(\norm{\xk-\xsol})	,
\]
which is much more straightforward than Theorem \ref{prop:linearisation}. 
The reason for such a difference is that the behaviour of deterministic FBS is monotonic, \eg $\norm{\xkp-\xsol} \leq \norm{\xk-\xsol}$, which allows us to encode all the small $o$-terms into $o(\norm{\xk-\xsol})$. 
\end{remark}

\subsubsection{No better rate estimation}

We discuss in short why the spectral radius of $\mFB$ cannot serve as the local convergence rate of stochastic optimisation methods, which is different from the deterministic setting. 
Let $K \in \bbN$ be sufficiently large such that \eqref{eq:linearised-inexact-fb} holds. 
Then we get
\beqn
\dkp
= \mFB^{k+1-K} d_{K}  +   \msum_{j=K}^{k} \mFB^{k-j}{{\phi_{j}}}  .
\eeqn
Take $\rho \in ]\rho(\mFB),1[$, owing to Lemma \ref{lem:eigmtxs}, there exists a constant $C > 0$ such that
\beqn
\begin{aligned}
\bbE(\norm{\dkp}) 
&\leq C {\rho}^{k+1-K} \bbE (\norm{x_{K}-\xsol}) + \msum_{j=K}^{k} \rho^{k-j}\bbE(\norm{\phi_{j}})  \\
&\leq C {\rho}^{k+1-K} \Ppa{ \bbE (\norm{x_{K}-\xsol}) + \rho^{K-1} \msum_{j=K}^{k} \qfrac{ \bbE(\norm{\phi_{j}})} { \rho^{j} } }   .
\end{aligned}
\eeqn
Now consider the SAGA algorithm, owing to Proposition \ref{prop:linear-rate-saga}, we have only that 
\[
\bbE(\norm{\phi_{j}}) = O\Pa{ (\sqrt{\rhosaga})^j }  ,
\]
which means $\lim_{k\to+\infty} \sum_{j=K}^{k} \frac{ \bbE(\norm{\phi_{j}})} { \rho^{j} } < +\infty$ holds only for $\rho \in ]\sqrt{\rhosaga}, 1[$. 
As a consequence, we can only obtain the same rate estimation as the original SAGA. 

\begin{remark}
The main message of the above discussion is: under a given step-size $\gamma$, the spectral radius $\rho(\mFB)$ is the optimal convergence rate can be achieved by SAGA/Prox-SVRG. However, depending on the problems to solve, the practical performance of these methods could be slower than $\rho(\mFB)$. 
\end{remark}

\paragraph{An overdetermined LASSO problem}
In Section \ref{sec:slr}, a sparse logistic regression problem is considered, where both SAGA and Prox-SVRG converge at the rate of $\rho(\mFB)$; see Figure \ref{fig:SLR}. 
Below we design an example of LASSO problem, to discuss the situations where $\rho(\mFB)$ cannot be achieved. 

Consider again the LASSO problem, 
\[
\min_{x\in\bbR^n} \mu\norm{x}_{1} + \qfrac{1}{m} \msum_{i=1}^{m} \sfrac{1}{2} \norm{\calK_{i} x - b_{i}}^2  ,	
\]
where now $\calK \in \bbR^{m\times n}$ is a random Gaussian matrix with zero means and $b \in \bbR^m$. 
Moreover, we choose $m=256, n = 32$, that is much more measurements than the size of the vector. 

For the test example, we have $L = 0.2239$ and the local quadratic grow parameter $\alpha = 0.0032$. The parameter choices of SAGA and Prox-SVRG with ``Option II'' are:
\[
\textrm{SAGA}:  \gamma = \qfrac{1}{3L}  ; \qquad
\textrm{Prox-SVRG}: \gamma = \qfrac{1}{10L} ,~~~ P = \qfrac{100L}{\alpha}  .
\]
We have $P \approx 27 m$ which is quite large. 
As discussion in the original work \cite{proxsvrg14}, with the above parameters choices, $\rhosvrg \approx \frac{5}{6}$.

The outcomes of the numerical experiments are shown in Figure \ref{fig:od-lasso}, where the observation of $\seq{\norm{\xk-\xsol}}$ is provided for SAGA and $\seq{\norm{\txl-\xsol}}$ for Prox-SVRG. The \emph{solid} lines stand for practical observations of the methods, the \emph{dashed} lines are the theoretical estimation from Proposition \ref{prop:linear-rate-saga} and \ref{prop:linear-rate-proxsvrg}, the \emph{dot-dashed} lines are the estimation from $\rho(\mFB)$. 
All the lines are sub-sampled, one out of every $m$ points for SAGA and $P$ points for Prox-SVRG. 
Note also that the observation is not in norm square. 

For this example, both the convergence speeds of SAGA and Prox-SVRG are slower than the spectral radius $\rho(\mFB)$. 
Empirically, the reason for SAGA is that the ratio of $m/n$ much larger than $1$, while for Prox-SVRG, the reason is that $P/m$ is too large.

\begin{figure}[!ht]
\centering
\subfloat[SAGA, $\norm{\xk-\xsol}$]{ \includegraphics[width=0.45\linewidth]{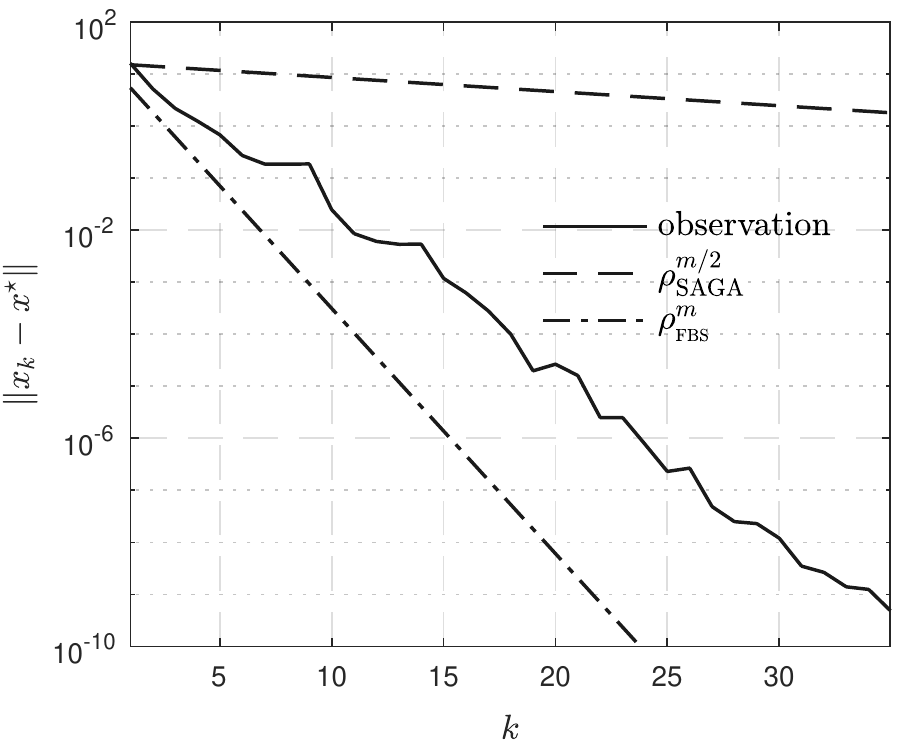} }  \hspace{2pt}
\subfloat[Prox-SVRG ``Option II'', $\norm{\txl-\xsol}$]{ \includegraphics[width=0.45\linewidth]{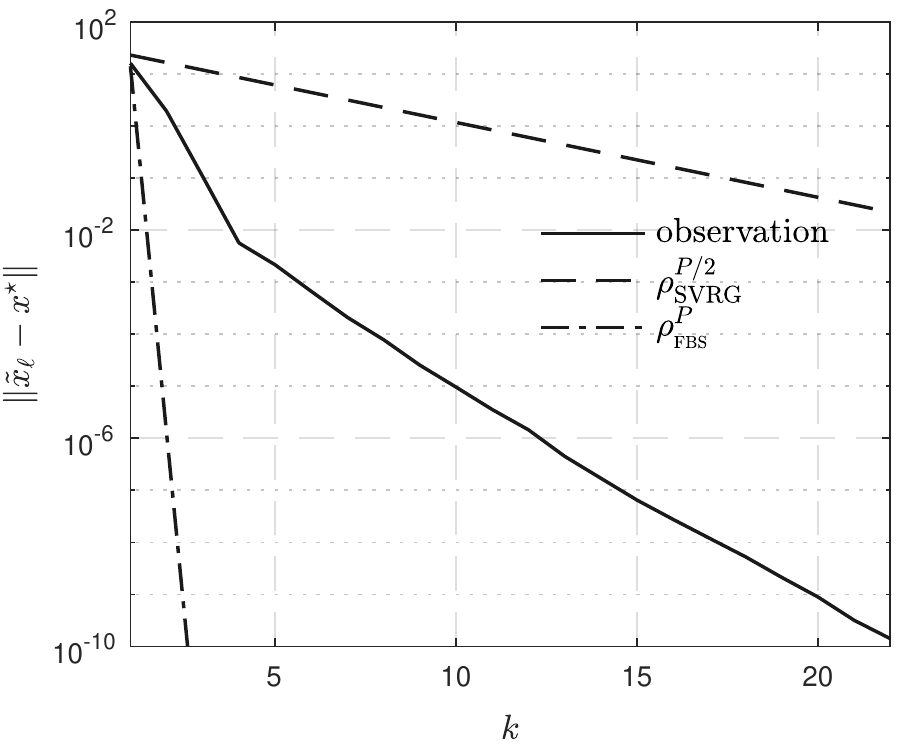} }  
\caption{
Convergence rate of SAGA and Prox-SVRG when solving an overdetermined LASSO problem. 
(a) convergence behaviour of $\norm{\xk-\xsol}$ of SAGA; (b) convergence behaviour of $\norm{\txl-\xsol}$ of Prox-SVRG.
The \emph{solid} lines stands for practical observations of the methods, the \emph{dashed} lines are the theoretical estimation from Proposition~\ref{prop:linear-rate-saga} and \ref{prop:linear-rate-proxsvrg}, the \emph{dot-dashed} lines are the estimation from $\rho(\mFB)$. 
All the lines are sub-sampled, one out of every $m$ points for SAGA and $P$ points for Prox-SVRG. $\rho_{_{\mFB}}$ denotes the spectral radius of $\mFB$.
}
 \label{fig:od-lasso}
\end{figure}

\subsection{Beyond local convergence analysis}\label{sec:beyond}

%
As already pointed out, manifold identification (Theorem \ref{thm:abstract-identification}) implies that, the globally non-smooth problem locally becomes a $C^2$-smooth and possibly non-convex (\eg nuclear norm) problem, constrained on the identified manifold, that is
\beqn
\begin{gathered} \textrm{$\min_{x\in\bbR^n} \Phi $} \\ \textrm{non\textrm{-}smooth on }~\bbR^n \end{gathered}
~\xRightarrow{~\textrm{Theorem~\ref{thm:abstract-identification}}~~~~}~
\begin{gathered} \textrm{$\min_{x\in\Msol} \Phi $} \\ \textrm{$C^2$-smooth on }~\Msol  \end{gathered}
\eeqn
Such a transition to local $C^2$-smoothness, provides various choices of acceleration. For instance, in \cite{lee2012manifold}, the authors proposed a local version of RDA, called RDA+, which achieves \emph{linear} convergence. 
In the following, we discuss several practical acceleration strategies.

\subsubsection{Better local Lipschitz continuity}

If the dimension of the manifold $\Msol$ is much smaller than that of the whole space $\bbR^n$, then constrained to $\Msol$, the Lipschitz property of the smooth part would become much better. 
For each $i\in\ba{1,\dotsm,m}$, denote by $L_{\Msol, i}$ the Lipschitz constant of $\nabla f_{i}$ along the manifold $\Msol$, and let 
\[
L_{\Msol} \eqdef \max_{i=1,\dotsm,m}  L_{\Msol, i}  .
\]
In general, locally around $\xsol$, we have $L_{\Msol} \leq L$. 

For SAGA/Prox-SVRG or other stochastic methods which have the manifold identification property, once the manifold is identified, they can adapt their step-sizes to the local Lipschitz of the problem {once the manifold is identified, one can adapt their step-sizes to the local Lipschitz constants of the problem}. Since step-size is crucial to the convergence speed of these algorithms, the potential acceleration of such as local adaptive strategy can be significant. 

In the numerical experiments section, this strategy is applied to the sparse logistic regression problem. For the considered problem, we have $L / L_{\Msol} \approx 16$, and the adaptive strategy achieves a $16$ times acceleration. 
It is worth mentioning that, the computational cost for evaluating $\Msol$ is negligible.

\subsubsection{Lower computational complexity}

Another important aspect of the {manifold identification property} is that {one can} reduce the computational cost, especially when $\Msol$ is of very low dimension.

Take $R = \norm{\cdot}_{1}$ as the $\ell_{1}$-norm for example. 
Suppose that the solution $\xsol$ of $\Phi$ is $\kappa$-sparse, \ie the number of non-zero entries of $\xsol$ is $\kappa$. 
We have two stages of gradient evaluation complexity for $\nabla f_{i}(\xk)$:
\begin{description}[leftmargin=5cm]
\item[Before identification]
$O(n)$ complexity;
\item[After identification]
$O(\kappa)$ complexity;
\end{description}
The reduction of computational cost is decided by the ratio of $n/\kappa$. Depending on this ratio, either mini-batch based methods, or even  deterministic methods with momentum acceleration can be applied (\eg inertial Forward--Backward schemes \cite{liang2017activity}, FISTA \cite{fista2009}).



\subsubsection{Higher-order acceleration}

The last acceleration strategy to discuss is the Riemannian manifold based higher-order acceleration. 
Recently, various the Riemannian manifold based optimisation methods are proposed in the literature \cite{kressner2014low,ring2012optimization,vandereycken2013low,boumal2014manopt}, particularly for low-rank matrix recovery.  
However, an obvious drawback of this class of methods is that the manifold should known \emph{a priori}, which limits the applications of these methods. 

The manifold identification property of proximal methods implies that one can first use the proximal method to identify the correct manifold, and then turn to the manifold based optimisation methods. The higher-order methods that can be applied include Newton-type method, when the restricted injectivity condition \eqref{eq:cnd-ri} is satisfied, and Riemannian geometry based optimisation methods \cite{LemarechalULagrangian,miller2005newton,smith1994optimization,boumal2014manopt,vandereycken2013low}, for instance the non-linear conjugate gradient method \cite{smith1994optimization}. 
Stochastic Riemannian manifold based optimisation methods are also studied in the literature, for instance in \cite{zhang2016riemannian}, the authors generalised the SVRG method to the manifold setting.

\section{Numerical experiments}\label{sec:experiment}

In this  section, we consider several concrete examples to illustrate our results. Three examples of $R$ are considered, sparsity promoting $\ell_{1}$-norm, group sparsity promoting $\ell_{1,2}$-norm and low rank promoting nuclear norm. We refer to \cite{liang2017activity} and the references therein for the detailed properties of these functionals. 

As the main focus of this work is the theoretical properties of SAGA and Prox-SVRG algorithms, the scale of the problems considered are not very large.

\subsection{Local linear convergence}
\label{sec:slr}

We consider the sparse logistic regression problem to demonstrate the manifold identification and local linear convergence of SAGA/Prox-SVRG algorithms. Moreover in this experiment, we provide only the rate estimation from the spectral radius $\rho(\mFB)$.

\begin{example}[Sparse logistic regression]\label{ex:slr}
Let $m > 0$ and $(z_i,y_i) \in \bbR^n \times \{\pm 1\} ,~ i=1,\cdots,m$ be the training set. 
The sparse logistic regression is to find a linear decision function which minimizes the objective
\beq\label{eq:hingesvm}
\min_{(x, b) \in \bbR^n \times \bbR } \mu \norm{x}_1 + \qfrac{1}{m} \msum_{i=1}^m \log\Pa{ 1+e^{ -y_{i} f(z_{i}; x,b) } }  ,
\eeq
where $f(z; x,b) = b + z^T x$.
\end{example}

The setting of the experiment is: $n = 256 ,~~~
m = 128 ,~~~    
\mu = 1/\sqrt{m} ~~\qandq %
L =  1188 $. 
Apparently, the dimension of the problem is larger than the number of training points. The parameters choices of SAGA and Prox-SVRG are:
\[
\textrm{SAGA}:  \gamma = \sfrac{1}{2L}  ; \quad
\textrm{Prox-SVRG}: \gamma = \sfrac{1}{3L} ,~~~ P = m  .
\]

\begin{remark}
The step-sizes of SAGA/Prox-SVRG exceeds the one allowed by Theorem \ref{thm:convergence-saga} and \ref{thm:convergence-proxsvrg}, respectively. The reason of choosing different step-sizes for SAGA and Prox-SVRG is mostly for the visual quality of the graphs in Figure \ref{fig:SLR}.  
\end{remark}

The observations of the experiments are shown in Figure~\ref{fig:SLR}. The observations of Prox-SVRG are for the inner loop sequence $x_{\ell, p}$, which is denoted as $\xk$ by letting $k= \ell P+p$. 
The non-degeneracy condition~\eqref{eq:cnd-nd} and the restricted injectivity condition \eqref{eq:cnd-ri} are checked \emph{a posterior}, which are all satisfied for the tested example. 
The local quadratic growth parameter $\alpha$ and the local Lipschitz constant $L_{\Msol}$ are
\[
\alpha =  0.0156
\qandq
L_{\Msol} = 61 .    
\]
Note that, locally the Lipschitz constant becomes about $19$ times better.

\paragraph{Finite manifold identification}
In Figure \ref{fig:SLR}(a), we plot the size of support of the sequence $\seq{\xk}$ generated by the two algorithms. The lines are sub-sampled, one out of every $m$ points.

The two algorithms are started with the same initial point. It is observed that SAGA shows faster manifold identification than Prox-SVRG, this is mainly due the fact that the step-size of SAGA (\ie $\gamma = \frac{1}{2L}$) is larger than that of Prox-SVRG (\ie $\gamma = \frac{1}{3L}$). 
The identification speed of the two algorithms are very close if they are applied under the same choice of step-size.

\begin{figure}[!ht]
\centering
\subfloat[$\abs{\supp(\xk)}$]{ \includegraphics[width=0.45\linewidth]{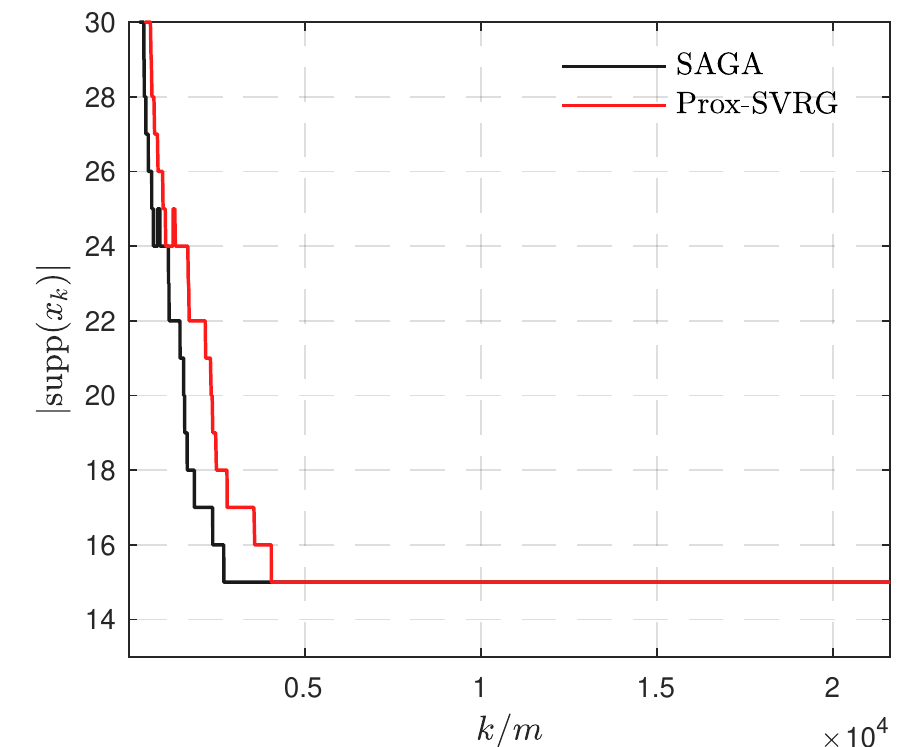} }  \hspace{2pt}
\subfloat[$\norm{\xk-\xsol}$]{ \includegraphics[width=0.45\linewidth]{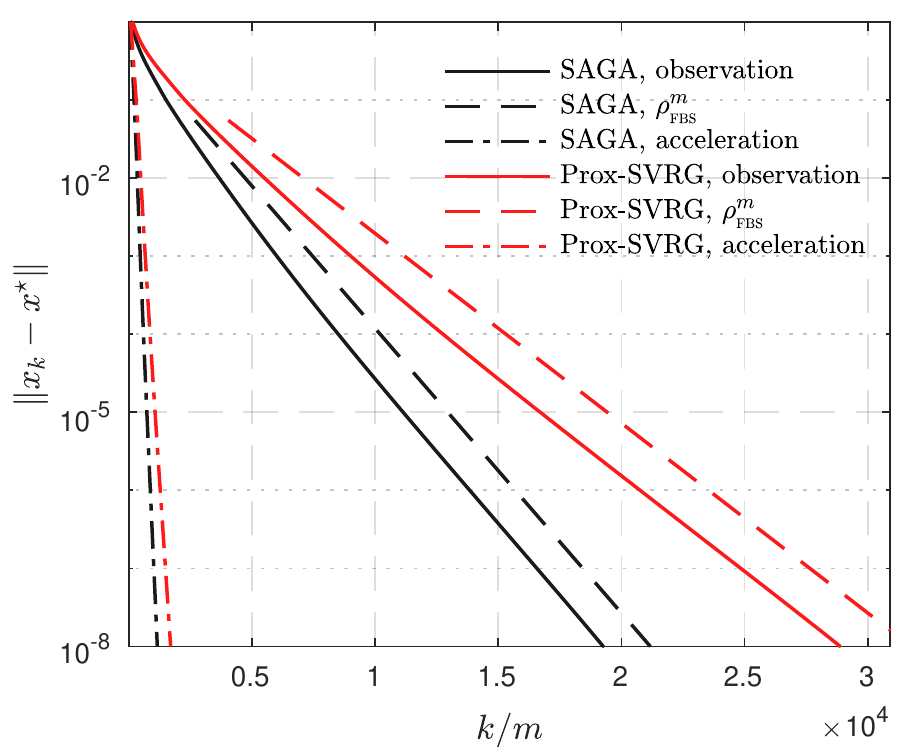} }  
\caption{
Finite manifold identification and local linear convergence of SAGA and Prox-SVRG for solving sparse logistic regression problem in Example \ref{ex:slr}. (a) finite manifold identification of SAGA/Prox-SVRG; (b) local linear convergence of SAGA/Prox-SVRG. $\rho_{_{\mFB}}$ denotes the spectral radius of $\mFB$.
}
 \label{fig:SLR}
\end{figure}

\paragraph{Local linear convergence}
In Figure \ref{fig:SLR}(b), we demonstrate the convergence rate of $\seq{\norm{\xk-\xsol}}$ of the two algorithms. The two \emph{solid} lines are the practical observation of $\seq{\norm{\xk-\xsol}}$ generated by SAGA and Prox-SVRG, the two \emph{dashed} lines are the theoretical estimations using the spectral radius of $\mFB$, and two \emph{dot-dashed} lines are the practical observation of the acceleration of SAGA/Prox-SVRG based on the local Lipschitz continuity $L_{\Msol}$. 
The lines are also sub-sampled, one out of every $m$ points. 

Since $\ell_{1}$-norm is polyhedral, the spectral radius of $\mFB$, denoted by $\rho_{_{\mFB}}$, is determined by $\alpha$ and $\gamma$, that is $\rho_{_{\mFB}} = 1 - \gamma\alpha$. 
Given the values of $\alpha$ and $\gamma$ of SAGA and Prox-SVRG, we have that
\[
\begin{aligned}
\textrm{SAGA}: &~~ \rho_{_{\mFB}} =  0.999993 ,~~ \rho_{_{\mFB}}^m = 0.99916  ; \\
\textrm{Prox-SVRG}: &~~ \rho_{_{\mFB}} =  0.999995  ,~~ \rho_{_{\mFB}}^m = 0.99944   . 
\end{aligned}
\]
For the consider problem setting, the spectral radius quite matches the practical observations.

To conclude this part, we highlight the benefits of adapting to the local Lipschitz continuity of the problem. For both SAGA and Prox-SVRG, their adaptive schemes (\eg \emph{dot-dashed} lines) shows $16$ times faster performance compared to the non-adaptive ones (\eg \emph{solid} lines). 
Such an acceleration gain is on the same order of the difference between the global Lipschitz and local Lipschitz constants, which is $19$ times. 
More importantly, the computational cost of {evaluating} the local Lipschitz {constant}  is almost {negligible}, which makes the adaptive scheme more preferable in practice.

\subsection{Local higher-order acceleration}

Now we consider two problems of group sparse and low-rank regression to demonstrate local higher-order acceleration. 

\begin{example}[Group sparse and low-rank regression \cite{Donoho06,candesExactCompletion}]\label{ex:gt-lasso}
Let $\xob \in \bbR^n$ be either a group sparse vector or a low-rank matrix (in a vectorised form), consider the following observation model
\[
b = \calK \xob + \omega ,    
\]
where the entries of $\calK \in \bbR^{m \times n}$ are sampled from i.i.d. zero-mean and unit-variance Gaussian distribution, $\omega \in \bbR^{m}$ is an additive error with bounded $\ell_{2}$-norm. 

Let $\mu > 0$, and $R(x)$ be either the group sparsity promoting $\ell_{1,2}$-norm or the low rank promoting nuclear norm. Consider the problem to recover or approximate $\xob$,
\beq\label{eq:g-t-lasso}
\min_{x \in \bbR^{n}} \mu R(x) + \qfrac{1}{m} \msum_{i=1}^{m} \sfrac{1}{2} \norm{\calK_{i} x - b_{i}}_{2} ^2 ,
\eeq
where $\calK_{i}, b_{i}$ represent the $i^\textrm{th}$ row and entry of $\calK$ and $b$, respectively. 
\end{example}

We have the following settings for the two examples of $R$:
\begin{description}[leftmargin=4.25cm]
\item[Group sparsity:]
$n = 512 ,~
m = 256$, $\xob$ has $8$ non-zero blocks of block-size $4$;
\item[Low rank:]
$n = 4096 ,~
m = 2048$, the rank of $\xob$ is $4$.
\end{description}
We consider only the SAGA algorithm for this test, as the main purpose is higher-order acceleration. For $\ell_{1,2}$-norm, Newton method is applied after the manifold identification, while for nuclear norm, a non-linear conjugate gradient~\cite{boumal2014manopt} is applied after manifold identification. 

The numerical results are shown in Figure \ref{fig:higher-order}. For $\ell_{1,2}$-norm, the black line is the observation of SAGA algorithm with $\gamma = \frac{1}{3L}$, the red line is the observation of ``SAGA+Newton'' hybrid scheme. It should be noted that the lines are not sub-sampled.

For the hybrid scheme, SAGA is used for manifold identification, and Newton method is applied once the manifold is identified. As observed, the quadratic convergence Newton method converges in only few steps. 
For nuclear norm, a non-linear conjugate gradient is applied when the manifold is identified. Similar to the observation of $\ell_{1,2}$-norm, the super-linearly convergent non-linear conjugate gradient shows superior performance to SAGA. 
%

\begin{figure}[!ht]
\centering
\subfloat[Group sparsity]{ \includegraphics[width=0.45\linewidth]{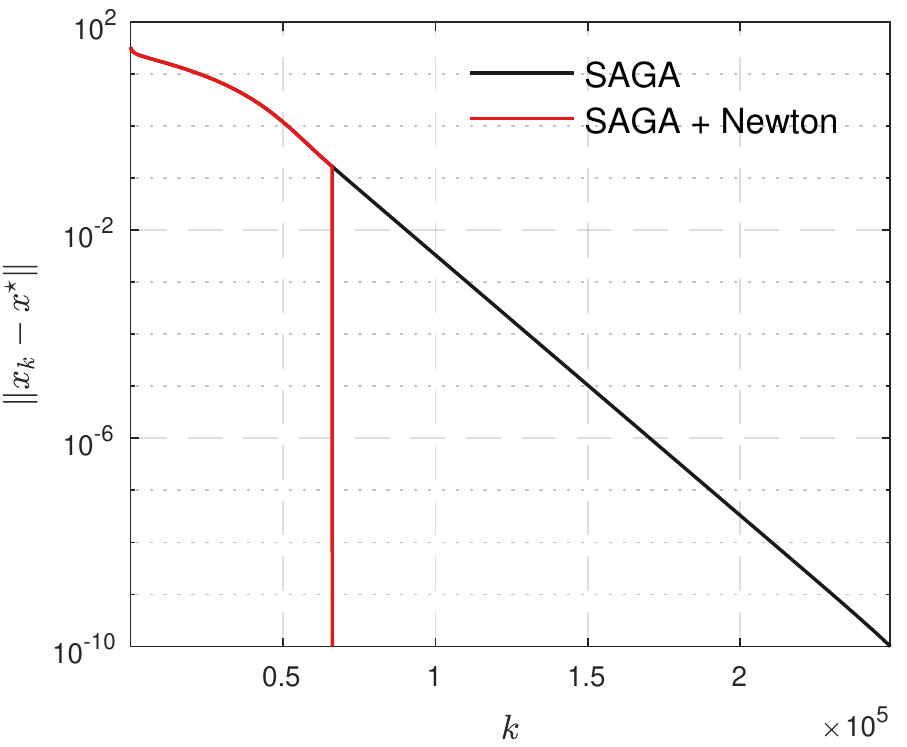} }  \hspace{2pt}
\subfloat[Low rank]{ \includegraphics[width=0.45\linewidth]{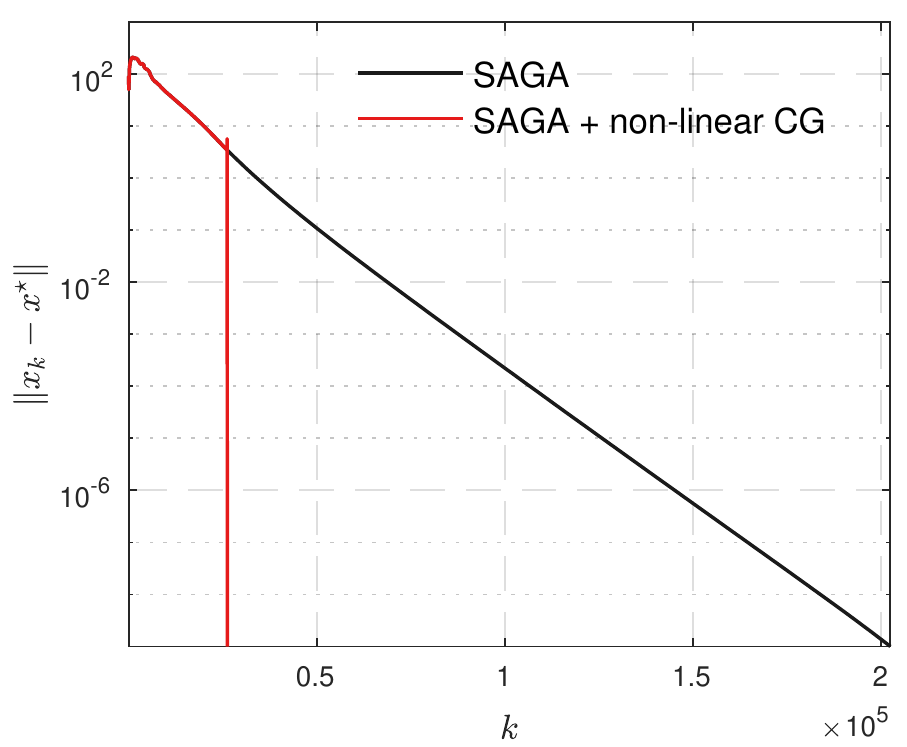} }  
\caption{
    Local higher-order acceleration after manifold identification in Example \ref{ex:gt-lasso}. (a) Newton method is applied after the manifold is identified by SAGA; (b) non-linear conjugate gradient is applied after manifold identification. Black line is the observation of SAGA algorithm, and the red line is the observation of SAGA+higher-order scheme. The black lines of SAGA for both examples are not sub-sampled. 
}
\label{fig:higher-order}
\end{figure}

\section{Conclusion}\label{sec:conclusion}

In this paper, we proposed a unified framework of local convergence analysis for proximal stochastic variance reduced gradient methods, and typically focused on SAGA and Prox-SVRG algorithms. 
Under partial smoothness, we established that these schemes identify the partial smooth manifold in finite time, and then converge locally linearly. Moreover, we proposed several practical acceleration approaches which can greatly improve the convergence speed of the algorithms.

\section*{Acknowledgements}
The authors would like to thank F. Bach, J. Fadili and G. Peyr\'{e} for helpful discussions.

\small
\appendix
\section{Proofs of theorems}
\label{appendix:proof}

\subsection{Proofs for Section \ref{sec:global-convergence}}
\label{appendix:proof-sec2}

To prove Theorem \ref{thm:convergence-saga} and \ref{thm:convergence-proxsvrg}, the lemma below is needed which is classical result from stochastic analysis \cite{neveu1975discrete}.

\begin{lemma}[Supermartingale convergence]\label{lem:supermartingale}
Let $Y_k$, $Z_k$ and $W_k$, $k=0,1,\ldots,$ be three sequences of random variables and let $\Ff_k$, $k=0,1,\ldots,$ be sets of random variables such that $\Ff_k\subset \Ff_{k+1}$ for all $k$. Suppose that:
\begin{enumerate}[label={\rm (\roman{*})}]
\item The random variables $Y_k$, $Z_k$ and $W_k$ are non-negative, and are functions of the random variables in $\Ff_k$.
\item For each $k$, we have $\EE(Y_{k+1}| \Ff_k) \leq Y_k - Z_k + W_k$.
\item With probability 1, $\sum_k W_k <\infty$.
\end{enumerate}
Then we have $\sum_k Z_k<\infty$ and the sequence $Y_k$ converges to a non-negative random variable $Y$ with probability 1.
\end{lemma}

\begin{proof}[Proof of Theorem \ref{thm:convergence-saga}]
The convergence of the objective function value for $\gamma_k \equiv \frac{1}{3L}$ is already studied in \cite{saga14}, here for the completeness of the proof, we shall keep the convergence proof of the objective function.

The proof of the theorem consists of several steps. First is the convergence of the objective function value. 
Let $\phi_{k,i}$ be the point such that $g_{k, i} = \nabla f_{i}(\phi_{k,i})$, then following the proof in the original SAGA paper \cite{saga14}, define the following Lyapunov function $\calL$,
\[
\calL_{k} 
\eqdef \calL(\xk, \ba{\phi_{k,i}}_{i=1}^{m})
\eqdef \qfrac{1}{m} \msum_{i=1}^{m} f_{i}(\phi_{k,i}) - F(\xsol) - \qfrac{1}{m} \msum_{i=1}^{m} \iprod{\nabla f_{i}(\xsol)}{\phi_{i, k}-\xsol} + c \norm{\xk - \xsol}^2
\]
for some appropriate $c > 0$. Denote $\bbE_{k}\sp{\cdot}$ the conditional expectation on step $k$. Then following the Appendix C of the supplementary material of \cite{saga14}, one can show that 
\beq\label{eq:calLk}
\bbE_{k}\sp{ \calL_{k+1} } \leq \calL_{k} - \qfrac{1}{4m} \bbE_{k}\sp{ \Phi(\xkp) - \Phi(\xsol) }  .
\eeq
Since $\bbE_{k}\sp{ \Phi(\xkp) - \Phi(\xsol) }$ is a non-negative random variable of the $k^\textrm{th}$ iteration, it then follows that $\seq{\calL_{k}}$ is a supermartingale owing to Lemma \ref{lem:supermartingale}. Therefore $\seq{\calL_{k}}$ converges to a non-negative random variable $\calL^\star$ with probability $1$. At the same time, with probability $1$, $\norm{\xk-\xsol}^2 \leq \frac{1}{c} \calL_{k}$, hence $\seq{\xk}$ is a bounded sequence and every cluster point of $\seq{\xk}$ is a global minimiser of $\Phi$. 
Moreover, from Lemma \ref{lem:supermartingale} and \eqref{eq:calLk}, we have
\[
\msum_{k=0}^{\infty} \Pa{ \bbE_{k}\sp{ \Phi(\xkp)-\Phi(\xsol) } } \leq \calL_{0} < \pinf  
\]
holds almost surely. 
Define a new random variable $y_{j} \eqdef \sum_{k\geq j} \bbE_{k} \sp{ \Phi(\xkp)-\Phi(\xsol) }$, clearly we have $\seq{y_{j}}$ is non-{in}creasing and converges to $0$ as $j \to \pinf$. As a consequence, {by the monotone convergence theorem,} we have
\[
0 
= \bbE \bSp{\lim_{j \to \pinf} y_{j} }
= \lim_{j \to \pinf} \bbE \sp{ y_{j} } 
= \lim_{j \to \pinf} \msum_{k \geq j} \bbE \sp{ \Phi(\xkp) - \Phi(\xsol) }
= \lim_{j \to \pinf} \bbE \bSp{ \msum_{k \geq j} ( \Phi(\xkp) - \Phi(\xsol) ) }  ,
\]
which implies
\beq\label{eq:sum-Phix-Phisol}
\bbE \bSp{ \msum_{k \geq j} \Pa{ \Phi(\xkp) - \Phi(\xsol) } } < \pinf
\Longrightarrow
\msum_{k} \Pa{ \Phi(\xkp) - \Phi(\xsol) } < \pinf ~~\textrm{almost surely}  ,
\eeq
hence $\Phi(\xk) \to \Phi(\xsol)$ almost surely.

With the boundedness of $\seq{\xk}$, the second step is to prove that $\seq{\norm{\xk-\xsol}}$ is convergent. Define a new sequence
\[
\begin{aligned}
\wk &\eqdef 	\qfrac{1}{m} \msum_{i=1}^{m} f_{i}(\phi_{k,i}) - F(\xsol) - \qfrac{1}{m} \msum_{i=1}^{m} \iprod{\nabla f_{i}(\xsol)}{\phi_{i, k}-\xsol} .  
\end{aligned}
\]
Observe that
\[
\bbE_{k} \sp{ \wkp } = \qfrac{1}{m} F(\xk) - F(\xsol) - \qfrac{1}{m} \iprod{\nabla F(\xsol)}{\xk-\xsol} + \Pa{1-\sfrac{1}{m}} \wk  .
\]
Since $\xsol \in \Argmin(\Phi)$ is a global minimiser, we have $-\nabla F(\xsol) \in \partial R(\xsol)$ and $\iprod{-\nabla F(\xsol)}{\xk-\xsol} \leq R(\xk) - R(\xsol)$, therefore from above equality we further obtain
\[
\bbE_{k} \sp{ \wkp } {\leq} \qfrac{1}{m} \Pa{\Phi(\xk)-\Phi(\xsol)} + \Pa{1-\sfrac{1}{m}} \wk  .
\]
Taking expectations over all previous steps for both sides and summing from $k=0$ to $j$ yields
\[
\bbE \sp{ {w_{j+1}} } + \qfrac{1}{m} \msum_{k=1}^{j} \bbE \sp{ \wk }
{\leq} \qfrac{1}{m} \msum_{k=0}^{j} \bbE \Sp{ \Phi(\xk)-\Phi(\xsol) } + \Pa{1-\sfrac{1}{m}} \bbE \sp{ w_{0} }  .
\]
As a result, taking $j$ to $\pinf$ implies that $\bbE \sp{ \sum_{k=1}^{j} \wk } < \pinf$, hence $\sum_{k=1}^{j} \wk < \pinf$ almost surely. Moreover, $\wk \to 0$ with probability $1$. From the convergence result of $\seq{\calL_{k}}$ and $\seq{\wk}$, we have that almost surely $\seq{\norm{\xk-\xsol}}$ is bounded and convergent.

 Next we prove the almost sure convergence of the sequence $\seq{\xk}$. 
Let $\ba{\xsol_{i}}_{i}$ be a countable subset of the relative interior $\ri(\Argmin(\Phi))$ that is dense in $\Argmin(\Phi)$. From the almost sure convergence of $\norm{\xk-\xsol}, \xsol \in \Argmin(\Phi)$, we have that for each $i$, the probability $\prob(\seq{\norm{\xk-\xsol_{i}}}~\textrm{is not convergent}) = 0$. Therefore
\[
\begin{aligned}
\prob\Pa{ \forall i, \exists b_{i}~ ~\st~ \lim_{k\to\pinf} \norm{\xk-\xsol_{i}} }
&= 1 - \prob(\seq{\norm{\xk-\xsol_{i}}}~\textrm{is not convergent}) \\
&\geq 1 - \msum_{i} \prob(\seq{\norm{\xk-\xsol_{i}}}~\textrm{is not convergent}) 
= 1  ,
\end{aligned}
\]
where the inequality follows from the union bound, \ie for each $i$, $\seq{\norm{\xk-\xsol_{i}}}$ is a convergent sequence. 
For a contradiction, suppose that there are convergent subsequences $\ba{u_{k_{j}}}_{k_{j}}$ and $\ba{v_{k_{j}}}_{k_{j}}$ of $\seq{\xk}$ which converge to their limiting points $\usol$ and $\vsol$ respectively, with $\norm{\usol-\vsol} = r > 0$. Since $\Phi(\xk)$ converges to $\inf \Phi$, these two limiting points are necessarily in $\Argmin(\Phi)$. Since $\ba{\xsol_{i}}_{i}$ is dense in $\Argmin(\Phi)$, we may assume that for all $\epsilon > 0$, we have $\xsol_{i_1}$ and $\xsol_{i_2}$ are such that $\norm{\xsol_{i_1}-\usol} < \epsilon$ and $\norm{\xsol_{i_2}-\vsol} < \epsilon$. Therefore, for all $k_{j}$ sufficiently large,
\[
\norm{ u_{k_{j}} - \xsol_{i_1} }
\leq \norm{u_{k_{j}} - \usol} + \norm{\usol + \xsol_{i_1}} 
< \norm{u_{k_{j}} - \usol} + \epsilon .
\]
On the other hand, for sufficiently large $j$, we have
\[
\norm{v_{k_{j}} - \xsol_{i_1}}
\geq \norm{\vsol-\usol} - \norm{\usol-\xsol_{i_1}} - \norm{v_{k_{j}} - \vsol }
> r - \epsilon  - \norm{v_{k_{j}} - \vsol } > r - 2\epsilon .
\]
This contradicts with the fact that $\xk- \xsol_{i_1}$ is convergent. Therefore, we must have $\usol=\vsol$, hence there exists $\xbar \in \Argmin(\Phi)$ such that $\xk \to \xbar$. 

Finally, to see that $\epsksaga\to 0$, from \cite[Lemma 6]{saga14},
$$
\qfrac{1}{m} \msum_{i=1}^m \norm{\nabla f_i(\phi_{k,i}) - \nabla f_i(x^*)}^2 \leq 2L w_k \to 0,
$$
therefore, combining this with the fact that $\nabla f_j$ is $L$-Lipschitz and $\xk \to x^*$, it follows that
$$
\norm{\epsksaga}\leq \norm{\nabla f_{\ik}(\xk) - \nabla f_{\ik}(\phi_{k,i})} + \qfrac{1}{m}\msum_{j=1}^m \norm{\nabla f_{j}(\phi_{k,i}) - \nabla f_j(\xk)} \to 0  ,
$$
which concludes the proof. \qedhere
\end{proof}

To prove Theorem \ref{thm:convergence-proxsvrg}, we require the following lemma,  which is a direct consequence of  Eq. (16) and Corollary~3 of \cite{proxsvrg14}. 
\begin{lemma}\label{lem:prox-svrg}
Assume that $F$ is $\alpha_F$-strongly convex and $R$ is $\alpha_R$-strongly convex. Let $\{x_{\ell,p}\}_{\ell,p}$ be the sequence generated by Prox-SVRG. Then, conditional on step $k=\ell P+p-1$, we have
\begin{equation}\label{eq:prox-svrg-bound}
\begin{split}
& (1+\gamma\alpha_{R})\EE_{k} \sp{ \norm{x_{\ell, p} - \xsol}^2 }  \\
&\leq (1-\gamma\alpha_{F}) \norm{x_{\ell, p-1} - \xsol}^2 - 2\gamma \Pa{ \Phi(x_{\ell, p}) - \Phi(\xsol) }  + 8 L \gamma^2\Ppa{ \Phi(x_{\ell, p-1}) - \Phi(\xsol) + \Phi(\txl) - \Phi(\xsol) }  .
\end{split}
\end{equation}
\end{lemma}

\begin{proof}[Proof of Theorem \ref{thm:convergence-proxsvrg}]
We begin with the remark that following the arguments in the proof of Theorem~\ref{thm:convergence-saga}, to show that $x_{\ell, p} \to \xsol$ almost surely for some $\xsol \in \argmin(\Phi)$, it is sufficient to prove that $\norm{x_{\ell, p} - \xsol}$ is convergent. 
By Lemma \ref{lem:prox-svrg} with $\alpha_R = \alpha_F = 0$, we have that conditional on step $k=\ell P+p-1$,
\begin{equation}\label{eq:svrg-1}
\EE_k \sp{ \norm{x_{\ell, p} - \xsol}^2 } + 2\gamma \EE_k \sp{ \Phi(x_{\ell, p}) - \Phi(\xsol) } 
\leq \norm{x_{\ell, p-1} - \xsol}^2 + 8 L \gamma^2\Ppa{ \Phi(x_{\ell, p-1}) - \Phi(\xsol) + \Phi(\txl) - \Phi(\xsol) }  .
\end{equation}
Summing \eqref{eq:svrg-1} over $p = 1, \ldots, P$ and taking expectation on the random variables $i_1,\ldots, i_P$,  we obtain that
\beq\label{eq:svrg-2}
\begin{aligned}
& \EE \sp{ \norm{x_{\ell, P} - \xsol }^2 } + 2\gamma \EE \sp{ \Phi(x_{\ell, P}) - \Phi(\xsol) } + 2\gamma(1-4L\gamma) \msum_{j=1}^{P-1} \EE \sp{ \Phi(x_{\ell, j}) - \Phi(\xsol ) }  \\
&\leq \norm{\tilde x_\ell - \xsol }^2 + 8L \gamma^2 (P+1) \Pa{ \Phi(\tilde x_{\ell}) - \Phi(\xsol ) }  .
\end{aligned}
\eeq
Since $\gamma \leq \frac{1}{ 4L  (P+2)}$, which yields $2\gamma(1-4L\gamma)  \geq \gamma^2$, we obtain from \eqref{eq:svrg-2}
\[
\begin{aligned}
& \EE \sp{\norm{x_{\ell, P} - \xsol }^2} + (2\gamma-\gamma^2) \EE \sp{ \Phi(x_{\ell, P}) - \Phi(x_*) } + \gamma^2 \msum_{j=1}^{P} \EE\sp{ \Phi(x_{\ell, j}) - \Phi(\xsol ) }  \\
&\leq \norm{\tilde x_{\ell} - \xsol }^2 + 8L \gamma^2 (P+1)(\Phi(\tilde x_{\ell}) - \Phi(\xsol )).
\end{aligned}
\]
Moreover, under ``Option I'', by defining the non-negative random variables 
\[
T_{\ell} \eqdef \norm{\tilde x_{\ell} - \xsol }^2 + (2\gamma-\gamma^2) (\Phi(\tilde x_{\ell}) - \Phi(\xsol )) 
\qandq 
S_{\ell+1} \eqdef \msum_{j=1}^{P} \Pa{ \Phi(x_{\ell, j}) - \Phi(\xsol ) }  .
\]
It follows from $ 8L \gamma^2 (P+1) \leq  2\gamma-\gamma^2$ that
\beq\label{eq:Tl}
\EE \sp{ T_{\ell+1} } \leq T_{\ell} - \gamma^2  \EE \sp{ S_{\ell+1} }  .
\eeq
So, by the super-martingale convergence theorem, $\ba{ T_{\ell} }_{\ell\in\bbN}$ converges to a non-negative random variable and $\sum_{\ell} S_{\ell} < +\infty$ holds almost surely. 
In particular, we have $S_{\ell} \to 0$ as $\ell \to \infty$ and hence, $\Phi(\tilde x_{\ell})\to \Phi(\xsol )$ as $\ell \to \infty$. Therefore, $\norm{\tilde x_{\ell} - \xsol }^2$ converges almost surely. Following the proof of Theorem \ref{thm:convergence-saga}, we can then show that $\tilde x_{\ell}$ converges to an optimal point $\xsol $ almost surely. 

Now we prove that the inner iteration sequence $\{x_{\ell,p}\}_{\substack{1\leq p \leq P,\; \ell \in \NN}}$ also converge to $\xsol$ as $\ell \to\infty$. Consider the inequality~\eqref{eq:svrg-1}, and define the non-negative random variables
\begin{equation}\label{eq:svrg3}
V_{\ell,p} \eqdef \norm{x_{\ell,p} - \xsol}^2 + 2\gamma \Pa{ \Phi(x_{\ell,p}) - \Phi(\xsol)  } 
\qandq
W_{\ell,p} \eqdef 8L\gamma^2 \Pa{ \Phi(\txl) - \Phi(\xsol) }  .
\end{equation}
Equation \eqref{eq:svrg-1} implies that
\[
\EE \sp{ V_{\ell,p} } \leq V_{\ell,p-1} + W_{\ell,p-1} ,
\] 
and moreover $\sum_{\ell,p} W_{\ell,p} = \sum_{\ell}S_{\ell} <\infty$ holds almost surely. 
Therefore, the super martingale convergence theorem implies that $\ba{ V_{\ell,p} }_{p \in \ba{1,\dotsm,P}, \ell \in \bbN}$ converges to a non-negative random variable. Moreover, since $\Phi(x_{\ell,p}) \to \Phi(\xsol)$, it follows that the sequence $\ba{ \norm{x_{\ell,p}- \xsol} }_{p \in \ba{1,\dotsm,P}, \ell \in \bbN}$ is convergent.

To prove the error rate \eqref{eq:svrg-rate}, observe that by convexity of $\Phi$ and Jensen's inequality, we have
\[
\EE \sp{ S_{\ell+1} } 
\geq P \EE  \Ssp{ \Phi\left(\sfrac{1}{P} \msum_{j=1}^P x_{\ell, j} \right)-\Phi(\xsol ) }  ,
\]
which further implies, owing to \eqref{eq:Tl},
\[
P\gamma^2 \EE \Ssp{ \Phi\left(\qfrac{1}{P} \msum_{j=1}^P x_{\ell, j} \right)-\Phi(\xsol ) } 
\leq \EE \sp{ T_\ell } - \EE \sp{ T_{\ell+1} }  .
\]
Summing over $\ell = 1,\ldots, Q$  and telescoping the right hand of the sum we arrive at
\[
\begin{aligned}
Q P\gamma^2 \EE \bSp{ \Phi \Pa{ \sfrac{1}{Q P} \msum_{\ell =1}^{Q} \msum_{j=1}^P x_{\ell, j} } - \Phi(\xsol ) }  
&\leq Q P\gamma^2 \EE \bSp{ \sfrac{1}{Q}\msum_{\ell=1}^{Q} \Phi\Ppa{ \sfrac{1}{P} \msum_{j=1}^P x_{\ell, j} } - \Phi(\xsol ) }   \\
&\leq \EE \sp{ T_{1} } - \EE \sp{ T_{ Q+1} }  ,
\end{aligned}
\]
where the first inequality follows from Jensen's inequality and convexity of $\Phi$. Dividing both sides by $kP\gamma^2$ gives the required error bound. The convergence of $\epsksvrg$ is a straightforward consequence of the convergence of $x_{l,p}$.

\vspace{4pt}

Now we prove the second claim of the theorem. 
Taking expectation of both sides of \eqref{eq:prox-svrg-bound} in Lemma \ref{lem:prox-svrg} and summing from $p=1,\dotsm,P$ yields
\[
\begin{aligned}
&(1-\gamma\alpha_{R})\EE\Sp{ \norm{x_{\ell,P}-\xsol}^2 } + 2\gamma\EE\Sp{ \Phi(x_{\ell,P}) - \Phi(\xsol) }  \\
&\leq -(\alpha_{F}+\alpha_{R}) \msum_{p=1}^{P} \EE\Sp{ \norm{x_{\ell,p}-\xsol}^2 } - \Pa{2\gamma - 8\gamma^2L} \msum_{p=1}^{P-1} \EE\Sp{ \Phi(x_{\ell,p}) - \Phi(\xsol) }  \\
&\qquad + (1-\gamma\alpha_{F}) \EE\Sp{ \norm{x_{\ell,0}-\xsol}^2 } + 8\gamma^2L\EE\Sp{ \Phi(x_{\ell,0}) - \Phi(\xsol) } + 8\gamma^2LP\EE\Sp{ \Phi(\txl) - \Phi(\xsol) }  .
\end{aligned}
\]
Since $\gamma L < \frac{1}{4(P+1)} < \frac{1}{4}$, we have $2\gamma - 8\gamma^2L > 0$, and we have from the above
\[
\begin{aligned}
(1-\gamma\alpha_{R})\EE\Sp{ \norm{x_{\ell,P}-\xsol}^2 } + 2\gamma\EE\Sp{ \Phi(x_{\ell,P}) - \Phi(\xsol) }  
\leq (1-\gamma\alpha_{F}) \EE\Sp{ \norm{x_{\ell,0}-\xsol}^2 }  + 8\gamma^2L(P+1)\EE\Sp{ \Phi(\txl) - \Phi(\xsol) }  .
\end{aligned}
\]
Define 
\[
T_{\ell}
\eqdef (1-\gamma\alpha_{R})\EE\Sp{ \norm{x_{\ell,P}-\xsol}^2 } + 2\gamma\EE\Sp{ \Phi(x_{\ell,P}) - \Phi(\xsol) }  ,
\]
then there holds
\[
\EE\sp{ T_{\ell} }    
\leq  \max\bBa{ \sfrac{1-\gamma\alpha_{F}}{1+\gamma\alpha_{R}}, 4L\gamma(P+1) }  \EE\sp{ T_{\ell-1} }    
\]
which implies the desired result. 
\qedhere
%
\end{proof}

\subsection{Proofs for Section \ref{sec:local-rate}}
\label{appendix:proof-sec4}

\subsubsection{Riemannian Geometry}
\label{sec:riemgeom}

Let $\calM$ be a $C^2$-smooth embedded submanifold of $\bbR^n$ around a point $x$. With some abuse of terminology, we shall state $C^2$-manifold instead of $C^2$-smooth embedded submanifold of $\bbR^n$. The natural embedding of a submanifold $\calM$ into $\bbR^n$ permits to define a Riemannian structure and to introduce geodesics on $\calM$, and we simply say $\calM$ is a Riemannian manifold. We denote respectively $\tanSp{\Mm}{x}$ and $\normSp{\Mm}{x}$ the tangent and normal space of $\Mm$ at point near $x$ in $\calM$.

\paragraph*{Exponential map}

Geodesics generalize the concept of straight lines in $\bbR^n$, preserving the zero acceleration characteristic, to manifolds. Roughly speaking, a geodesic is locally the shortest path between two points on $\calM$. We denote by $\mathfrak{g}(t;x, h)$ the value at $t \in \bbR$ of the geodesic starting at $\mathfrak{g}(0;x,h) = x \in \calM$ with velocity $\dot{\mathfrak{g}}(t;x, h) = \qfrac{d\mathfrak{g}}{dt}(t;x,h) = h \in \tanSp{\Mm}{x}$ (which is uniquely defined).  
For every $h \in \tanSp{\calM}{x}$, there exists an interval $I$ around $0$ and a unique geodesic $\mathfrak{g}(t;x, h): I \to \calM$ such that $\mathfrak{g}(0; x, h) = x$ and $\dot{\mathfrak{g}}(0;x, h) = h$.
The mapping
 \[
\Exp_x 
: \tanSp{\calM}{x} \to  \calM ,~~   h\mapsto \Exp_{x}(h) =  \mathfrak{g}(1;x, h) ,
\]
is called \emph{Exponential map}.
Given $x, x' \in \calM$, the direction $h \in \tanSp{\calM}{x}$ we are interested in is such that 
\[
\Exp_x(h) = x' = \mathfrak{g}(1;x, h)  .
\]

\paragraph*{Parallel translation}

Given two points $x, x' \in \calM$, let $\tanSp{\calM}{x}, \tanSp{\calM}{x'}$ be their corresponding tangent spaces. Define 
\[
\tau : \tanSp{\calM}{x} \to \tanSp{\calM}{x'} ,
\]
the parallel translation along the unique geodesic joining $x$ to $x'$, which is isomorphism and isometry w.r.t. the Riemannian metric.

\paragraph*{Riemannian gradient and Hessian}

For a vector $v \in \normSp{\calM}{x}$, the Weingarten map of $\calM$ at $x$ is the operator $\Wgtmap{x}\pa{\cdot, v}: \tanSp{\calM}{x} \to \tanSp{\calM}{x}$ defined by
\[
\Wgtmap{x}\pa{\cdot, v} = - \PT{\tanSp{\calM}{x}} \mathrm{d} V[h]  ,
\]
where $V$ is any local extension of $v$ to a normal vector field on $\calM$. The definition is independent of the choice of the extension $V$, and $\Wgtmap{x}(\cdot, v)$ is a symmetric linear operator which is closely tied to the second fundamental form of $\calM$, see \cite[Proposition II.2.1]{chavel2006riemannian}.

Let $G$ be a real-valued function which is $C^2$ along the $\calM$ around $x$. The covariant gradient of $G$ at $x' \in \calM$ is the vector $\nabla_{\calM} G(x') \in \tanSp{\calM}{x'}$ defined by
\[
\iprod{\nabla_{\calM} G(x')}{h} = \qfrac{d}{dt} G\Pa{\PT{\calM}(x'+th)}\big|_{t=0} ,~~ \forall h \in \tanSp{\calM}{x'},
\]
where $\PT{\calM}$ is the projection operator onto $\calM$. 
The covariant Hessian of $G$ at $x'$ is the symmetric linear mapping $\nabla^2_{\calM} G(x')$ from $\tanSp{\calM}{x'}$ to itself which is defined as
\beq\label{eq:rh}
\iprod{\nabla^2_{\calM} G(x') h}{h} = \qfrac{d^2}{dt^2} G\Pa{\PT{\calM}(x'+th)}\big|_{t=0} ,~~ \forall h \in \tanSp{\calM}{x'} .
\eeq
This definition agrees with the usual definition using geodesics or connections \cite{miller2005newton}. 
Now assume that $\calM$ is a Riemannian embedded submanifold of $\bbR^{n}$, and that a function $G$ has a $C^2$-smooth restriction on $\calM$. This can be characterized by the existence of a $C^2$-smooth extension (representative) of $G$, \ie a $C^2$-smooth function $\widetilde{G}$ on $\bbR^{n}$ such that $\widetilde{G}$ agrees with $G$ on $\calM$. Thus, the Riemannian gradient $\nabla_{\calM}G(x')$ is also given by
\beq\label{eq:RieGradient}
\nabla_{\calM} G(x') = \PT{\tanSp{\calM}{x'}} \nabla \widetilde{G}(x')  ,
\eeq
and $\forall h \in \tanSp{\calM}{x'}$, the Riemannian Hessian reads
\beq\label{eq:RieHessian}
\begin{aligned}
\nabla^2_{\calM} G(x') h
&= \PT{\tanSp{\calM}{x'}} \mathrm{d} (\nabla_{\calM} G)(x')[h]
= \PT{\tanSp{\calM}{x'}} \mathrm{d} \Pa{ x' \mapsto \PT{\tanSp{\calM}{x'}} \nabla_{\calM} \widetilde{G} }[h]   \\ 
&= \PT{\tanSp{\calM}{x'}} \nabla^2 \widetilde{G}(x') h + \Wgtmap{x'}\Pa{h, \PT{\normSp{\calM}{x'}}\nabla \widetilde{G}(x')}  ,
\end{aligned}
\eeq
where the last equality comes from  \cite[Theorem~1]{absil2013extrinsic}. 
When $\calM$ is an affine or linear subspace of $\bbR^{n}$, then obviously $\calM = x + \tanSp{\calM}{x}$, and $\Wgtmap{x'}\pa{h, \PT{\normSp{\calM}{x'}} \nabla \widetilde{G}(x')} = 0$, hence \eqref{eq:RieHessian} reduces to
\[
\nabla^2_{\calM} G(x')
= \PT{\tanSp{\calM}{x'}} \nabla^2 \widetilde{G}(x')  \PT{\tanSp{\calM}{x'}}  .
\]
See \cite{lee2003smooth,chavel2006riemannian} for more materials on differential and Riemannian manifolds.

The following lemmas summarize two key properties that we will need throughout.
\begin{lemma}[{\cite[Lemma B.1]{liang2017activity}}]\label{lem:parallel-translation}
Let $x \in \calM$, and $\xk$ a sequence converging to $x$ in $\calM$. Denote $\tau_k : \tanSp{\calM}{x} \to \tanSp{\calM}{\xk}$ be the parallel translation along the unique geodesic joining $x$ to $\xk$. Then, for any bounded vector $u \in \bbR^n$, we have
\[
\pa{\tau_k^{-1}\proj_{\tanSp{\calM}{\xk}} - \proj_{\tanSp{\calM}{x}}}u = o(\norm{u})  .
\]
\end{lemma}

\begin{lemma}[{\cite[Lemma B.2]{liang2017activity}}]\label{lem:taylor-expn}
Let $x, x'$ be two close points in $\calM$, denote $\tau : \tanSp{\calM}{x} \to \tanSp{\calM}{x'}$ the parallel translation along the unique geodesic joining $x$ to $x'$. The Riemannian Taylor expansion of $\Phi \in C^2(\calM)$ around $x$ reads,
\[\label{eq:taylor-expn}
\tau^{-1} \nabla_{\calM} \Phi(x') = \nabla_{\calM} \Phi(x) + \nabla^2_{\calM} \Phi(x)\proj_{\tanSp{\calM}{x}}(x'-x) + o(\norm{x'-x})  .
\]
\end{lemma}

\begin{lemma}[Local normal sharpness {\cite[Proposition~2.10]{Lewis-PartlySmooth}}]
\label{fact:sharp}
If $R \in \PSF{x}{\calM}$, then all $x' \in \calM$ near $x$ satisfy $\tanSp{\calM}{x'} = T_{x'}$. In particular, when $\calM$ is affine or linear, then
$T_{x'} = T_{x}$.
\end{lemma}

Next we provide expressions of the Riemannian gradient and Hessian for the case of partly smooth functions relative to a $C^2$-smooth submanifold. This is summarized in the following proposition which follows by combining Eq. \eqref{eq:RieGradient} and \eqref{eq:RieHessian}, Definition~\ref{dfn:psf}, Lemma~\ref{fact:sharp} and \cite[Proposition~17]{Daniilidis06} (or \cite[Lemma~2.4]{miller2005newton}).

\begin{lemma}[Riemannian gradient and Hessian]\label{lem:gradhess}
If $R \in \PSF{x}{\calM}$, then for any $x' \in \calM$ near $x$
\[
\nabla_{\calM} R(x') = \proj_{T_{x'}}\pa{\partial R(x')} .
\]
For all $h \in T_{x'}$,
\[
\nabla^2_{\calM} R(x')h = \proj_{T_{x'}} \nabla^2 \widetilde{R}(x')h + \Wgtmap{x'}\Pa{h, \proj_{T_{x'}^\perp}\nabla \widetilde{R}(x')}  ,
\]
where $\widetilde{R}$ is a smooth representation of $R$ on $\calM$, and $\Wgtmap{x}\pa{\cdot, \cdot}: T_x \times T_x^\perp \to T_x$ is the Weingarten map of $\calM$ at $x$. 
\end{lemma}

\subsubsection{Proofs}

\begin{proof}[Proof of Proposition~\ref{prop:linearisation}]
By virtue the definition of proximity operator and the update of $\xkp$ in \eqref{eq:inexact-fbs}, we have 
\[
\xk - \xkp - \gammak \Pa{\nabla F(\xk) - \nabla F(\xkp)} - \gammak\epsk \in \gammak \partial \Phi(\xkp) . 
\]
Given a global minimiser $\xsol \in \Argmin(\Phi)$, the classic optimality condition entails that
\[
0 \in \gammak \partial \Phi(\xsol)  .
\]
Projecting the above two inclusions on to $T_{\xkp}$ and $\Tsol$, respectively and using Lemma~\ref{lem:gradhess}, lead to
\[
\begin{aligned}
\gammak \tau_{k+1}^{-1}  \nabla_{\Msol} \Phi(\xkp)  
&= \tau_{k+1}^{-1}\PT{T_{\xkp}} \Pa{\xk - \xkp - \gammak \pa{\nabla F(\xk) - \nabla F(\xkp)} - \gammak\epsk } \\
\gammak \nabla_{\Msol} \Phi(\xsol)  &=  0  .
\end{aligned}
\]
Adding both identities, and subtracting $\tau_{k+1}^{-1}\PT{T_{\xkp}}\xsol$ on both sides, we get
\beq\label{eq:inter-step-1}
\begin{aligned}
& \tau_{k+1}^{-1}\PT{T_{\xkp}} (\xkp-\xsol) + \gammak \Pa{\tau_{k+1}^{-1}\nabla_{\Msol} \Phi(\xkp) - \nabla_{\Msol} \Phi(\xsol)}  \\
&= \tau_{k+1}^{-1}\PT{T_{\xkp}}\pa{\xk - \xsol} - \gammak \tau_{k+1}^{-1}\PT{T_{\xkp}} \Pa{\nabla F(\xk) - \nabla F(\xkp)}  - { \gammak \tau_{k+1}^{-1}\PT{T_{\xkp}} \epsk }  .
\end{aligned}
\eeq
For each term of \eqref{eq:inter-step-1}, we have the following result
\begin{enumerate}[label={\rm (\roman{*})}]
\item
By virtue of Lemma~\ref{lem:parallel-translation}, we get
\[
\begin{aligned}
\tau_{k+1}^{-1} \PT{T_{\xkp}} \pa{\xkp-\xsol} 
&= \PT{\Tsol} \pa{\xkp-\xsol} + \pa{\tau_{k+1}^{-1} \PT{T_{\xkp}} - \PT{\Tsol}} \pa{\xkp-\xsol}  \\
&= \PT{\Tsol} \pa{\xkp-\xsol}  + o(\norm{\xkp-\xsol}) .
\end{aligned}
\]
With the help of \cite[Lemma~5.1]{liang2014local}, that $\xkp-\xsol = \PT{\Tsol} \pa{\xkp-\xsol} + o(\norm{\xkp-\xsol})$, we further derive
\beq\label{eq:PT-xkp-xsol}
\tau_{k+1}^{-1} \PT{T_{\xkp}} \pa{\xkp-\xsol} 
 = \pa{\xkp-\xsol}  + o(\norm{\xkp-\xsol}) . 
\eeq
Similarly for $\xk$, we have $\tau_{k+1}^{-1} \PT{T_{\xkp}}\pa{\xk - \xsol} =  \pa{\xk - \xsol} + o(\norm{\xk - \xsol})$. 
%
\item
Owing to Lemma \ref{lem:taylor-expn}, we have for $\tau_{k+1}^{-1}\nabla_{\Msol} \Phi(\xkp) - \nabla_{\Msol} \Phi(\xsol)$,
\beq\label{eq:hess-stuff}
\tau_{k+1}^{-1}\nabla_{\Msol} \Phi(\xkp) - \nabla_{\Msol} \Phi(\xsol)
= \nabla^2_{\Msol} \Phi(\xsol)\proj_{\Tsol} (\xkp-\xsol) + o(\norm{\xkp-\xsol})  .
\eeq
\item
Using Lemma \ref{lem:taylor-expn} again together with the local $C^2$-smoothness of $F$, we have
\beq\label{eq:gradF-sth-1}
\begin{aligned}
&\tau_{k+1}^{-1}\PT{T_{\xkp}}\Pa{\nabla F(\xk) - \nabla F(\xkp)} \\
&= \PT{\Tsol}\Pa{\nabla F(\xk) - \nabla F(\xkp)} + o(\norm{\nabla F(\xk) - \nabla F(\xkp)}) \\
&= \PT{\Tsol} \Pa{ \pa{ \nabla F(\xk) - \nabla F(\xsol)} - \pa{ \nabla F(\xkp) - \nabla F(\xsol)} }  + o(\norm{\nabla F(\xk) - \nabla F(\xsol)} + \norm{\nabla F(\xkp) - \nabla F(\xsol)}) \\
&= \PT{\Tsol} \nabla^2 F(\xsol) \PT{\Tsol}  (\xk-\xsol) - \PT{\Tsol} \nabla^2 F(\xsol) \PT{\Tsol}  (\xkp-\xsol) + o(\norm{\xk-\xsol}) + o(\norm{\xkp-\xsol})  .
\end{aligned}
\eeq
\item
Owing to Lemma~\ref{lem:parallel-translation}, we have ${  \tau_{k+1}^{-1}\PT{T_{\xkp}} \epsk }
= {  \PT{\Tsol} \epsk  + o(\norm{\epsk}) }$. 
\end{enumerate}
Combining the above relations with \eqref{eq:inter-step-1} we obtain
\beq\label{eq:inter-step-2}
\begin{aligned}
& \Pa{ \Id + \gammak \nabla^2_{\Msol} {\Phi}(\xsol)\proj_{\Tsol} - \gammak \PT{\Tsol}\nabla^2 F(\xsol)\PT{\Tsol} } (\xkp-\xsol) + o(\norm{\xkp-\xsol}) \\
&= \pa{ \Id + \gammak \HR } (\xkp-\xsol) + o(\norm{\xkp-\xsol})  \\
&=\pa{ \Id + \gamma \HR } (\xkp-\xsol) + (\gammak-\gamma) \HR (\xkp-\xsol) + o(\norm{\xkp-\xsol})  \\
&= \pa{\xk - \xsol} - \gammak \HF  (\xk-\xsol) + o(\norm{\xk-\xsol})   - \pa{ \gammak \PT{\Tsol} \epsk  + o(\norm{\epsk}) }   \\
&= \pa{\xk - \xsol} - \gamma \HF  (\xk-\xsol) - (\gammak-\gamma) \HF  (\xk-\xsol) + o(\norm{\xk-\xsol})   - \pa{ \gamma \PT{\Tsol} \epsk  + o(\norm{\epsk}) } - (\gammak-\gamma) \PT{\Tsol} \epsk  .
\end{aligned}
\eeq
Since we have $\gammak \to \gamma$ and $\HR$ is bounded, we have
\beqn
\lim_{k\to+\infty} \qfrac{ \norm{(\gammak-\gamma) \HR (\xkp-\xsol)} }{\norm{\xkp-\xsol}} 
\leq \lim_{k\to+\infty} \qfrac{ \abs{\gammak-\gamma} \norm{\HR} \norm{\xkp-\xsol} }{\norm{\xkp-\xsol}} 
= 0  ,
\eeqn
hence $(\gammak-\gamma) \HR (\xkp-\xsol) = o(\norm{\xkp-\xsol})$. 
Using the same arguments lead to $(\gammak-\gamma) \HF  (\xk-\xsol) = o(\norm{\xk-\xsol})$ and $(\gammak-\gamma) \PT{\Tsol} \epsk = o(\norm{\epsk})$.
Therefore, from \eqref{eq:inter-step-2}, we obtain
\beq\label{eq:inter-step-3}
\begin{aligned}
\pa{\Id + \gamma \HR} (\xkp-\xsol) + o(\norm{\xkp-\xsol})  
= \GF \pa{\xk -  \xsol}  + o(\norm{\xk-\xsol})   - \pa{ \gamma \PT{\Tsol} \epsk  + o(\norm{\epsk}) }   .
\end{aligned}
\eeq
Inverting $\Id + \gamma \HR$ (which is possible owing to Lemma~\ref{lem:eigmtxs}), we obtain
\beq\label{eq:inter-step-4} 
\begin{aligned}
\xkp-\xsol + \WR o(\norm{\xkp-\xsol})
&= \WR \GF \pa{\xk - \xsol} + \WR o(\norm{\xk-\xsol})   - \WR \pa{ \gamma \PT{\Tsol} \epsk  + o(\norm{\epsk}) }   \\
&= \mFB \pa{\xk - \xsol} + \WR o(\norm{\xk-\xsol})   -   \WR \pa{  \gamma\PT{\Tsol} \epsk  + o(\norm{\epsk}) }   .
\end{aligned}
\eeq
Since $\WR, \GF$ are non-expansive, we have
\[
\xkp-\xsol + \WR o(\norm{\xkp-\xsol})
= \xkp-\xsol +  o(\norm{\xkp-\xsol})
= \dkp  ,
\]
and similarly, we have $\mFB \pa{\xk - \xsol} + \WR o(\norm{\xk-\xsol}) = \mFB \pa{\xk - \xsol} +  o(\norm{\xk-\xsol}) = \mFB \dk$ and $\gamma \WR \PT{\Tsol} \epsk  + \WR o(\norm{\epsk}) = \gamma \WR \PT{\Tsol} \epsk  +  o(\norm{\epsk}) = \phi_{k}$. Substituting back into \eqref{eq:inter-step-4} we conclude the proof. 
\qedhere
\end{proof}

\begin{small}
\bibliographystyle{plain}
\bibliography{bib}

\begin{thebibliography}{10}

\bibitem{absil2009optimization}
P-A. Absil, R.~Mahony, and R.~Sepulchre.
\newblock {\em Optimization algorithms on matrix manifolds}.
\newblock Princeton University Press, 2009.

\bibitem{absil2013extrinsic}
P-A. Absil, R.~Mahony, and J.~Trumpf.
\newblock An extrinsic look at the {R}iemannian {H}essian.
\newblock In {\em Geometric Science of Information}, pages 361--368. Springer,
  2013.

\bibitem{agarwal2012fast}
A.~Agarwal, S.~Negahban, and M.~J. Wainwright.
\newblock Fast global convergence of gradient methods for high-dimensional
  statistical recovery.
\newblock {\em The Annals of Statistics}, 40(5):2452--2482, 2012.

\bibitem{alvarez2000minimizing}
F.~Alvarez.
\newblock On the minimizing property of a second order dissipative system in
  {H}ilbert spaces.
\newblock {\em SIAM Journal on Control and Optimization}, 38(4):1102--1119,
  2000.

\bibitem{alvarez2001inertial}
F.~Alvarez and H.~Attouch.
\newblock An inertial proximal method for maximal monotone operators via
  discretization of a nonlinear oscillator with damping.
\newblock {\em Set-Valued Analysis}, 9(1-2):3--11, 2001.

\bibitem{AttouchFISTAGradInexact15}
H.~Attouch, Z.~Chbani, J.~Peypouquet, and P.~Redont.
\newblock Fast convergence of inertial dynamics and algorithms with asymptotic
  vanishing damping.
\newblock Technical Report Optimization online 5179, 2015.

\bibitem{AttouchFISTA15}
H.~Attouch and J.~Peypouquet.
\newblock The rate of convergence of {N}esterov's accelerated
  {F}orward--{B}ackward method is actually $o(k^{-2})$.
\newblock Technical Report arXiv:1510.08740, 2015.

\bibitem{AttouchiFB14}
H.~Attouch, J.~Peypouquet, and P.~Redont.
\newblock A dynamical approach to an inertial {F}orward--{B}ackward algorithm
  for convex minimization.
\newblock {\em SIAM J. Optim.}, 24(1):232--256, 2014.

\bibitem{AttouchFISTAGrad15}
H.~Attouch, J.~Peypouquet, and P.~Redont.
\newblock On the fast convergence of an inertial gradient-like dynamics with
  vanishing viscosity.
\newblock Technical Report arXiv:1507.04782, 2015.

\bibitem{DossalAujol15}
J.-F. Aujol and C.~Dossal.
\newblock Stability of over-relaxations for the {F}orward--{B}ackward
  algorithm, application to fista.
\newblock {\em SIAM Journal on Optimization}, 25(4):2408--2433, 2015.

\bibitem{baillon1977quelques}
J.~B. Baillon and G.~Haddad.
\newblock Quelques propri{\'e}t{\'e}s des op{\'e}rateurs angle-born{\'e}s
  etn-cycliquement monotones.
\newblock {\em Israel Journal of Mathematics}, 26(2):137--150, 1977.

\bibitem{bauschke2011convex}
H.~Bauschke and P.~L. Combettes.
\newblock {\em Convex Analysis and Monotone Operator Theory in Hilbert Spaces}.
\newblock Springer, 2011.

\bibitem{beck2009fast}
A.~Beck and M.~Teboulle.
\newblock Fast gradient-based algorithms for constrained total variation image
  denoising and deblurring problems.
\newblock {\em Image Processing, IEEE Transactions on}, 18(11):2419--2434,
  2009.

\bibitem{fista2009}
A.~Beck and M.~Teboulle.
\newblock A fast iterative shrinkage-thresholding algorithm for linear inverse
  problems.
\newblock {\em SIAM Journal on Imaging Sciences}, 2(1):183--202, 2009.

\bibitem{BolteKL06}
J.~Bolte, A.~Daniilidis, and A.~Lewis.
\newblock The {{\L}}ojasiewicz inequality for nonsmooth subanalytic functions
  with applications to subgradient dynamical systems.
\newblock {\em SIAM J. Optim}, 17(4):1205--1223, 2006.

\bibitem{bredies2008linear}
K.~Bredies and D.~A. Lorenz.
\newblock Linear convergence of iterative soft-thresholding.
\newblock {\em Journal of Fourier Analysis and Applications}, 14(5-6):813--837,
  2008.

\bibitem{BurachikBook08}
Regina~S. Burachik and Alfredo~N. Iusem.
\newblock {\em Set-valued Mappings and Enlargements of Monotone Operators}.
\newblock Optimization and Its Applications. Springer, 2008.

\bibitem{candes2013simple}
E.~J. Cand{\`e}s and B.~Recht.
\newblock Simple bounds for recovering low-complexity models.
\newblock {\em Mathematical Programming}, 141(1-2):577--589, 2013.

\bibitem{chambolle2015convergence}
A.~Chambolle and C.~Dossal.
\newblock On the convergence of the iterates of the ``fast iterative
  shrinkage/thresholding algorithm''.
\newblock {\em Journal of Optimization Theory and Applications},
  166(3):968--982, 2015.

\bibitem{chavel2006riemannian}
I.~Chavel.
\newblock {\em Riemannian geometry: a modern introduction}, volume~98.
\newblock Cambridge University Press, 2006.

\bibitem{combettes2001quasi}
P.~L. Combettes.
\newblock Quasi-{F}ej{\'e}rian analysis of some optimization algorithms.
\newblock {\em Studies in Computational Mathematics}, 8:115--152, 2001.

\bibitem{combettes2014variable}
P.~L. Combettes and B.~C. V{\~u}.
\newblock Variable metric {F}orward--{B}ackward splitting with applications to
  monotone inclusions in duality.
\newblock {\em Optimization}, 63(9):1289--1318, 2014.

\bibitem{Daniilidis-SpectralIdent}
A.~Daniilidis, D.~Drusvyatskiy, and A.~S. Lewis.
\newblock Orthogonal invariance and identifiability.
\newblock {\em SIAM Journal on Matrix Analysis and Applications},
  35(2):580--598, 2014.

\bibitem{Daniilidis06}
A.~Daniilidis, W.~Hare, and J.~Malick.
\newblock Geometrical interpretation of the predictor-corrector type algorithms
  in structured optimization problems.
\newblock {\em Optimization: A Journal of Mathematical Programming \&
  Operations Research}, 55(5-6):482--503, 2009.

\bibitem{duval2015sparse}
V.~Duval and G.~Peyr{\'e}.
\newblock Sparse spikes deconvolution on thin grids.
\newblock {\em arXiv preprint arXiv:1503.08577}, 2015.

\bibitem{goldstein2014fast}
T.~Goldstein, B.~O'Donoghue, S.~Setzer, and R.~Baraniuk.
\newblock Fast alternating direction optimization methods.
\newblock {\em SIAM Journal on Imaging Sciences}, 7(3):1588--1623, 2014.

\bibitem{Gu2012}
M.~Gu, L.-H. Lim, and C.~J. Wu.
\newblock {P}ar{N}es: a rapidly convergent algorithm for accurate recovery of
  sparse and approximately sparse signals.
\newblock {\em Numerical Algorithms}, 64(2):321--347, 2012.

\bibitem{hale07}
E.~Hale, W.~Yin, and Y.~Zhang.
\newblock Fixed-point continuation for $\ell_1$-minimization: methodology and
  convergence.
\newblock {\em SIAM Journal on Optimization}, 19(3):1107--1130, 2008.

\bibitem{HareFB11}
W.~L. Hare.
\newblock Identifying active manifolds in regularization problems.
\newblock In H.~H. Bauschke, R.~S., Burachik, P.~L. Combettes, V.~Elser, D.~R.
  Luke, and H.~Wolkowicz, editors, {\em Fixed-Point Algorithms for Inverse
  Problems in Science and Engineering}, volume~49 of {\em Springer Optimization
  and Its Applications}, chapter~13. Springer, 2011.

\bibitem{hare2004identifying}
W.~L. Hare and A.~S. Lewis.
\newblock Identifying active constraints via partial smoothness and
  prox-regularity.
\newblock {\em Journal of Convex Analysis}, 11(2):251--266, 2004.

\bibitem{Hare-Lewis-Algo}
W.~L. Hare and A.~S. Lewis.
\newblock Identifying active manifolds.
\newblock {\em Algorithmic Operations Research}, 2(2):75--82, 2007.

\bibitem{hiriart1996convex}
J.-B. Hiriart-Urruty and C.~Lemar{\'e}chal.
\newblock {\em Convex Analysis And Minimization Algorithms}, volume I and II.
\newblock Springer, 2001.

\bibitem{hou2013linear}
K.~Hou, Z.~Zhou, A.~M.-C. So, and Z.Q. Luo.
\newblock On the linear convergence of the proximal gradient method for trace
  norm regularization.
\newblock In {\em Advances in Neural Information Processing Systems}, pages
  710--718, 2013.

\bibitem{johnstone2015lyapunov}
P.~R. Johnstone and P.~Moulin.
\newblock A {L}yapunov analysis of {FISTA} with local linear convergence for
  sparse optimization.
\newblock {\em arXiv preprint arXiv:1502.02281}, 2015.

\bibitem{lee2003smooth}
J.~M. Lee.
\newblock {\em Smooth manifolds}.
\newblock Springer, 2003.

\bibitem{Lemarechal-ULagrangian}
C.~Lemar{\'e}chal, F.~Oustry, and C.~Sagastiz{\'a}bal.
\newblock The {U}-{L}agrangian of a convex function.
\newblock {\em Trans. Amer. Math. Soc.}, 352(2):711--729, 2000.

\bibitem{Lewis-PartlySmooth}
A.~S. Lewis.
\newblock Active sets, nonsmoothness, and sensitivity.
\newblock {\em SIAM Journal on Optimization}, 13(3):702--725, 2003.

\bibitem{LewisPartlyTiltHessian}
A.~S. Lewis and S.~Zhang.
\newblock Partial smoothness, tilt stability, and generalized {H}essians.
\newblock {\em SIAM Journal on Optimization}, 23(1):74--94, 2013.

\bibitem{liang2014local}
J.~Liang, J.~Fadili, and G.~Peyr{\'e}.
\newblock Local linear convergence of {F}orward--{B}ackward under partial
  smoothness.
\newblock In {\em Advances in Neural Information Processing Systems}, pages
  1970--1978, 2014.

\bibitem{lions1979splitting}
P.~L. Lions and B.~Mercier.
\newblock Splitting algorithms for the sum of two nonlinear operators.
\newblock {\em SIAM Journal on Numerical Analysis}, 16(6):964--979, 1979.

\bibitem{lorenz2014accelerated}
D.~A. Lorenz and T.~Pock.
\newblock An accelerated {F}orward--{B}ackward algorithm for monotone
  inclusions.
\newblock {\em arXiv preprint arXiv:1403.3522}, 2014.

\bibitem{miller2005newton}
S.~A. Miller and J.~Malick.
\newblock Newton methods for nonsmooth convex minimization: connections
  among-{L}agrangian, {R}iemannian {N}ewton and {SQP} methods.
\newblock {\em Mathematical programming}, 104(2-3):609--633, 2005.

\bibitem{mordukhovich1992sensitivity}
B.S. Mordukhovich.
\newblock Sensitivity analysis in nonsmooth optimization.
\newblock {\em Theoretical Aspects of Industrial Design (D. A. Field and V.
  Komkov, eds.), SIAM Volumes in Applied Mathematics}, 58:32--46, 1992.

\bibitem{moudafi2003convergence}
A.~Moudafi and M.~Oliny.
\newblock Convergence of a splitting inertial proximal method for monotone
  operators.
\newblock {\em Journal of Computational and Applied Mathematics},
  155(2):447--454, 2003.

\bibitem{Nesterov83}
Y.~Nesterov.
\newblock A method for solving the convex programming problem with convergence
  rate {$O(1/k^2)$}.
\newblock {\em Dokl. Akad. Nauk SSSR}, 269(3):543--547, 1983.

\bibitem{nesterov2004introductory}
Y.~Nesterov.
\newblock {\em Introductory lectures on convex optimization: A basic course},
  volume~87.
\newblock Springer, 2004.

\bibitem{nesterov2007gradient}
Y.~Nesterov.
\newblock Gradient methods for minimizing composite objective function.
\newblock 2007.

\bibitem{o2012adaptive}
B.~O'Donoghue and E.~Candes.
\newblock Adaptive restart for accelerated gradient schemes.
\newblock {\em Foundations of computational mathematics}, 15(3):715--732, 2015.

\bibitem{opial1967weak}
Z.~Opial.
\newblock Weak convergence of the sequence of successive approximations for
  nonexpansive mappings.
\newblock {\em Bulletin of the American Mathematical Society}, 73(4):591--597,
  1967.

\bibitem{poliquin1998tilt}
R.~A. Poliquin and R.~T. Rockafellar.
\newblock Tilt stability of a local minimum.
\newblock {\em SIAM Journal on Optimization}, 8(2):287--299, 1998.

\bibitem{Polyack64}
B.~T. Polyack.
\newblock Some methods of speeding up the convergence of iterative methods.
\newblock {\em Zh. Vychisl. Mat. Mat. Fiz.}, 4:1--17, 1964.

\bibitem{polyak1987introduction}
B.~T. Polyak.
\newblock {\em Introduction to optimization}.
\newblock Optimization Software, 1987.

\bibitem{rockafellar1976monotone}
R.~T. Rockafellar.
\newblock Monotone operators and the proximal point algorithm.
\newblock {\em SIAM Journal on Control and Optimization}, 14(5):877--898, 1976.

\bibitem{rockafellar1998variational}
R.~T. Rockafellar and R.~Wets.
\newblock {\em Variational analysis}, volume 317.
\newblock Springer Verlag, 1998.

\bibitem{rudin1992nonlinear}
L.~I. Rudin, S.~Osher, and E.~Fatemi.
\newblock Nonlinear total variation based noise removal algorithms.
\newblock {\em Physica D: Nonlinear Phenomena}, 60(1):259--268, 1992.

\bibitem{smith1994optimization}
S.~T. Smith.
\newblock Optimization techniques on {R}iemannian manifolds.
\newblock {\em Fields institute communications}, 3(3):113--135, 1994.

\bibitem{SvaiterBurachik99}
B.~F. Svaiter and R.~S. Burachik.
\newblock $\varepsilon$-enlargements of maximal monotone operators in banach
  spaces.
\newblock {\em Set-Valued Anal.}, 7:117--132, 1999.

\bibitem{tao2015local}
S.~Tao, D.~Boley, and S.~Zhang.
\newblock Local linear convergence of {ISTA} and {FISTA} on the {LASSO}
  problem.
\newblock {\em arXiv preprint arXiv:1501.02888}, 2015.

\bibitem{Tseng09}
P.~Tseng and S.~Yun.
\newblock A coordinate gradient descent method for nonsmooth separable
  minimization.
\newblock {\em Math. Prog. (Ser. B)}, 117, 2009.

\bibitem{Vaiterdof15}
S.~Vaiter, C.~Deledalle, J.~M. Fadili, G.~Peyr\'e, and C.~Dossal.
\newblock The degrees of freedom of partly smooth regularizers.
\newblock {\em Annals of the Institute of Mathematical Statistics}, 2015.
\newblock to appear.

\bibitem{vaiter2015model}
S.~Vaiter, M.~Golbabaee, J.~Fadili, and G.~Peyr{\'e}.
\newblock Model selection with low complexity priors.
\newblock {\em Information and Inference}, page iav005, 2015.

\bibitem{vaiter2014modelconsistency}
S.~Vaiter, G.~Peyr{\'e}, and J.~Fadili.
\newblock Model consistency of partly smooth regularizers.
\newblock {\em arXiv preprint arXiv:1405.1004}, 2014.

\bibitem{Wright-IdentSurf}
S.~J. Wright.
\newblock Identifiable surfaces in constrained optimization.
\newblock {\em SIAM Journal on Control and Optimization}, 31(4):1063--1079,
  1993.

\end{thebibliography}


\begin{thebibliography}{10}

\bibitem{absil2009optimization}
P-A. Absil, R.~Mahony, and R.~Sepulchre.
\newblock {\em Optimization algorithms on matrix manifolds}.
\newblock Princeton University Press, 2009.

\bibitem{absil2013extrinsic}
P-A. Absil, R.~Mahony, and J.~Trumpf.
\newblock An extrinsic look at the {R}iemannian {H}essian.
\newblock In {\em Geometric Science of Information}, pages 361--368. Springer,
  2013.

\bibitem{fista2009}
A.~Beck and M.~Teboulle.
\newblock A fast iterative shrinkage-thresholding algorithm for linear inverse
  problems.
\newblock {\em SIAM Journal on Imaging Sciences}, 2(1):183--202, 2009.

\bibitem{blatt2007convergent}
D.~Blatt, A.~O. Hero, and H.~Gauchman.
\newblock A convergent incremental gradient method with a constant step size.
\newblock {\em SIAM Journal on Optimization}, 18(1):29--51, 2007.

\bibitem{boumal2014manopt}
N.~Boumal, B.~Mishra, P.-A. Absil, R.~Sepulchre, et~al.
\newblock Manopt, a matlab toolbox for optimization on manifolds.
\newblock {\em Journal of Machine Learning Research}, 15(1):1455--1459, 2014.

\bibitem{candesExactCompletion}
E.~J. Cand\`es and B.~Recht.
\newblock Exact matrix completion via convex optimization.
\newblock {\em Foundations of Computational Mathematics}, 9(6):717--772, 2009.

\bibitem{chambolle2015convergence}
A.~Chambolle and C.~Dossal.
\newblock On the convergence of the iterates of the ``fast iterative
  shrinkage/thresholding algorithm''.
\newblock {\em Journal of Optimization Theory and Applications},
  166(3):968--982, 2015.

\bibitem{chavel2006riemannian}
I.~Chavel.
\newblock {\em Riemannian geometry: a modern introduction}, volume~98.
\newblock Cambridge University Press, 2006.

\bibitem{combettes2005signal}
P.~L. Combettes and V.~R. Wajs.
\newblock Signal recovery by proximal {F}orward--{B}ackward splitting.
\newblock {\em Multiscale Modeling \& Simulation}, 4(4):1168--1200, 2005.

\bibitem{Daniilidis06}
A.~Daniilidis, W.~Hare, and J.~Malick.
\newblock Geometrical interpretation of the predictor-corrector type algorithms
  in structured optimization problems.
\newblock {\em Optimization: A Journal of Mathematical Programming \&
  Operations Research}, 55(5-6):482--503, 2009.

\bibitem{defazio2014new}
A.~Defazio.
\newblock {\em New Optimisation Methods for Machine Learning}.
\newblock PhD thesis, Australian National University, 2014.
\newblock http://www.aarondefazio.com/pubs.html.

\bibitem{saga14}
A.~Defazio, F.~Bach, and S.~Lacoste-Julien.
\newblock Saga: A fast incremental gradient method with support for
  non-strongly convex composite objectives.
\newblock In {\em Advances in Neural Information Processing Systems}, pages
  1646--1654, 2014.

\bibitem{Donoho06}
D.~L. Donoho, M.~Elad, and V.~N. Temlyakov.
\newblock Stable recovery of sparse overcomplete representations in the
  presence of noise.
\newblock {\em IEEE Trans. Inform. Theory}, 52(1):6--18, 2006.

\bibitem{duchi2016local}
J.~Duchi and F.~Ruan.
\newblock Local asymptotics for some stochastic optimization problems:
  Optimality, constraint identification, and dual averaging.
\newblock {\em arXiv preprint arXiv:1612.05612}, 2016.

\bibitem{fadili2017sensitivity}
J.~Fadili, J.~Malick, and G.~Peyr{\'e}.
\newblock Sensitivity analysis for mirror-stratifiable convex functions.
\newblock {\em arXiv preprint arXiv:1707.03194}, 2017.

\bibitem{gong2014linear}
P.~Gong and J.~Ye.
\newblock Linear convergence of variance-reduced stochastic gradient without
  strong convexity.
\newblock {\em arXiv preprint arXiv:1406.1102}, 2014.

\bibitem{hare2004identifying}
W.~L. Hare and A.~S. Lewis.
\newblock Identifying active constraints via partial smoothness and
  prox-regularity.
\newblock {\em Journal of Convex Analysis}, 11(2):251--266, 2004.

\bibitem{Hare-Lewis-Algo}
W.~L. Hare and A.~S. Lewis.
\newblock Identifying active manifolds.
\newblock {\em Algorithmic Operations Research}, 2(2):75--82, 2007.

\bibitem{svrg}
R.~Johnson and T.~Zhang.
\newblock Accelerating stochastic gradient descent using predictive variance
  reduction.
\newblock In {\em Advances in neural information processing systems}, pages
  315--323, 2013.

\bibitem{kressner2014low}
D.~Kressner, M.~Steinlechner, and B.~Vandereycken.
\newblock Low-rank tensor completion by riemannian optimization.
\newblock {\em BIT Numerical Mathematics}, 54(2):447--468, 2014.

\bibitem{roux2012stochastic}
N.~Le~Roux, M.~Schmidt, and F.~R. Bach.
\newblock A stochastic gradient method with an exponential convergence \_rate
  for finite training sets.
\newblock In {\em Advances in Neural Information Processing Systems}, pages
  2663--2671, 2012.

\bibitem{lee2003smooth}
J.~M. Lee.
\newblock {\em Smooth manifolds}.
\newblock Springer, 2003.

\bibitem{lee2012manifold}
S.~Lee and S.~J. Wright.
\newblock Manifold identification in dual averaging for regularized stochastic
  online learning.
\newblock {\em Journal of Machine Learning Research}, 13(Jun):1705--1744, 2012.

\bibitem{LemarechalULagrangian}
C.~Lemar{\'e}chal, F.~Oustry, and C.~Sagastiz{\'a}bal.
\newblock The {U}-{L}agrangian of a convex function.
\newblock {\em Trans. Amer. Math. Soc.}, 352(2):711--729, 2000.

\bibitem{Lewis-PartlySmooth}
A.~S. Lewis.
\newblock Active sets, nonsmoothness, and sensitivity.
\newblock {\em SIAM Journal on Optimization}, 13(3):702--725, 2003.

\bibitem{LewisPartlyTiltHessian}
A.~S. Lewis and S.~Zhang.
\newblock Partial smoothness, tilt stability, and generalized {H}essians.
\newblock {\em SIAM Journal on Optimization}, 23(1):74--94, 2013.

\bibitem{liang2014local}
J.~Liang, J.~Fadili, and G.~Peyr{\'e}.
\newblock Local linear convergence of {F}orward--{B}ackward under partial
  smoothness.
\newblock In {\em Advances in Neural Information Processing Systems}, pages
  1970--1978, 2014.

\bibitem{liang2014convergence}
J.~Liang, J.~Fadili, and G.~Peyr{\'{e}}.
\newblock Convergence rates with inexact non-expansive operators.
\newblock {\em Mathematical Programming}, 159(1):403--434, September 2016.

\bibitem{liang2017activity}
J.~Liang, J.~Fadili, and G.~Peyr{\'e}.
\newblock Activity identification and local linear convergence of
  {F}orward--{B}ackward-type methods.
\newblock {\em SIAM Journal on Optimization}, 27(1):408--437, 2017.

\bibitem{lions1979splitting}
P.~L. Lions and B.~Mercier.
\newblock Splitting algorithms for the sum of two nonlinear operators.
\newblock {\em SIAM Journal on Numerical Analysis}, 16(6):964--979, 1979.

\bibitem{lorenz2015inertial}
D.~A. Lorenz and T.~Pock.
\newblock An inertial forward-backward algorithm for monotone inclusions.
\newblock {\em Journal of Mathematical Imaging and Vision}, 51(2):311--325,
  2015.

\bibitem{miller2005newton}
S.~A. Miller and J.~Malick.
\newblock Newton methods for nonsmooth convex minimization: connections
  among-{L}agrangian, {R}iemannian {N}ewton and {SQP} methods.
\newblock {\em Mathematical programming}, 104(2-3):609--633, 2005.

\bibitem{molinari2018convergence}
C.~Molinari, J.~Liang, and J.~Fadili.
\newblock Convergence rates of forward--douglas--rachford splitting method.
\newblock {\em arXiv preprint arXiv:1801.01088}, 2018.

\bibitem{moudafi2003convergence}
A.~Moudafi and M.~Oliny.
\newblock Convergence of a splitting inertial proximal method for monotone
  operators.
\newblock {\em Journal of Computational and Applied Mathematics},
  155(2):447--454, 2003.

\bibitem{nesterov2004introductory}
Y.~Nesterov.
\newblock {\em Introductory lectures on convex optimization: A basic course},
  volume~87.
\newblock Springer, 2004.

\bibitem{neveu1975discrete}
J.~Neveu.
\newblock {\em Discrete-parameter martingales}, volume~10.
\newblock Elsevier, 1975.

\bibitem{ring2012optimization}
W.~Ring and B.~Wirth.
\newblock Optimization methods on riemannian manifolds and their application to
  shape space.
\newblock {\em SIAM Journal on Optimization}, 22(2):596--627, 2012.

\bibitem{sag17}
M.~Schmidt, N.~Le~Roux, and F.~Bach.
\newblock Minimizing finite sums with the stochastic average gradient.
\newblock {\em Mathematical Programming}, 162(1-2):83--112, 2017.

\bibitem{smith1994optimization}
S.~T. Smith.
\newblock Optimization techniques on {R}iemannian manifolds.
\newblock {\em Fields institute communications}, 3(3):113--135, 1994.

\bibitem{vandereycken2013low}
B.~Vandereycken.
\newblock Low-rank matrix completion by riemannian optimization.
\newblock {\em SIAM Journal on Optimization}, 23(2):1214--1236, 2013.

\bibitem{vanli2016global}
N.~D. Vanli, M.~Gurbuzbalaban, and A.~Ozdaglar.
\newblock Global convergence rate of proximal incremental aggregated gradient
  methods.
\newblock {\em arXiv preprint arXiv:1608.01713}, 2016.

\bibitem{xiao2010dual}
L.~Xiao.
\newblock Dual averaging methods for regularized stochastic learning and online
  optimization.
\newblock {\em Journal of Machine Learning Research}, 11(Oct):2543--2596, 2010.

\bibitem{proxsvrg14}
L.~Xiao and T.~Zhang.
\newblock A proximal stochastic gradient method with progressive variance
  reduction.
\newblock {\em SIAM Journal on Optimization}, 24(4):2057--2075, 2014.

\bibitem{zhang2016riemannian}
H.~Zhang, S.~J. Reddi, and S.~Sra.
\newblock Riemannian svrg: Fast stochastic optimization on riemannian
  manifolds.
\newblock In {\em Advances in Neural Information Processing Systems}, pages
  4592--4600, 2016.

\end{thebibliography}
\end{small}

\end{document}